\documentclass{amsart}
\usepackage{latexsym,amsxtra,amscd,ifthen}
\usepackage{amsfonts}
\usepackage{verbatim}
\usepackage{amsmath}
\usepackage{amsthm}
\usepackage{amssymb}
\usepackage{url}
\usepackage{color}
\usepackage{tikz}

\usetikzlibrary{calc}

\theoremstyle{plain}

\numberwithin{equation}{section}

\newtheorem{theorem}[equation]{Theorem}
\newtheorem{lemma}[equation]{Lemma}

\theoremstyle{definition}
\newtheorem{definition}[equation]{Definition}

\newtheorem{remark}[equation]{Remark}
\newtheorem{question}[equation]{Question}
\newcommand{\assign}{:=}

\newcommand{\Z}{\mathbb{Z}}
\newcommand{\N}{\mathbb{N}}
\newcommand{\F}{\mathbb{F}}

\DeclareMathOperator{\End}{End}
\DeclareMathOperator{\tr}{tr}
\DeclareMathOperator{\Aut}{Aut}
\DeclareMathOperator{\gr}{gr}

\newcommand{\bfx}{\mathbf x}
\newcommand{\bfs}{\mathbf s}
\newcommand{\bft}{\mathbf t}
\newcommand{\bfe}{\mathbf e}
\newcommand{\bfa}{\mathbf a}
\newcommand{\bfb}{\mathbf b}

\begin{document}

\title[Discriminants and Automorphism Groups of Veronese subrings] 
{Discriminants and Automorphism Groups\\
of Veronese subrings of skew polynomial rings}

\author{K. Chan, A.A. Young, and J.J. Zhang}

\address{(Chan) Department of Mathematics, Box 354350,
University of Washington, Seattle, Washington 98195, USA}
\email{kenhchan@math.washington.edu}

\address{(Young) Department of Mathematics,
DigiPen Institute of Technology, Redmond, WA 98052, USA}
\email{young.mathematics@gmail.com}

\address{(Zhang) Department of Mathematics, Box 354350,
University of Washington, Seattle, Washington 98195, USA}
\email{zhang@math.washington.edu}

\begin{abstract}
We study important invariants and  properties of 
the Veronese subalgebras of $q$-skew polynomial rings, 
including their discriminant, center and automorphism 
group, as well as cancellation property and the Tits alternative. 
\end{abstract}





\maketitle


\section*{Introduction}
\label{xxsec0}

The determination of the full automorphism group of an algebra
is a fundamental problem in mathematics. This is generally extremely 
difficult,  
for example, even for the \emph{polynomial ring in three variables}
its automorphism group is not well understood. Aside from a 
remarkable result of Shestakov-Umirbaev 
\cite{SU} which shows that the Nagata automorphism is a wild 
automorphism, the general structure of this automorphism group eludes our grasp. 

Since the 1990s, researchers have successfully computed the full
automorphism group of several interesting families of noncommutative 
algebras of finite Gelfand-Kirillov dimension, including certain 
quantum groups, generalized quantum Weyl algebras, skew polynomial 
rings -- see \cite{AlC, AlD, AnD, BJ, GTK, GY, LL, SAV}. A few 
years ago, by using a rigidity theorem for quantum tori, Yakimov proved 
the Andruskiewitsch-Dumas conjecture \cite{Y1} and the Launois-Lenagan 
conjecture \cite{Y2}, each of which determines the full automorphism 
group of an important class of quantized algebras. Recently, 
Ceken-Palmieri-Wang and the third-named author introduced 
a discriminant method to control the automorphism group of certain 
classes of algebras \cite{CPWZ1, CPWZ2} and then were able to compute 
the automorphism group of several more families of Artin-Schelter regular 
algebras that satisfy a polynomial identity. The wisdom behind much of this
progress is that noncommutative algebras are more rigid, so that more techniques are
available for detecting their symmetries. 

In most of the results mentioned above, the algebras are deformations 
(in some weak sense) of the commutative polynomial rings. 
In this paper we apply the discriminant method to certain noncommutative
algebras that are \emph{not} deformations of polynomial rings. We are mainly 
interested in the automorphism problem, but will briefly touch upon the 
cancellation problem and the Tits alternative.

Throughout the introduction let $\Bbbk$ denote our base field
and $\Bbbk^{\times}$ be the its group of units.
For a $\Bbbk$-algebra $A$, let $\Aut(A)$ denote the group of 
$\Bbbk$-algebra automorphisms of $A$.

Fix a $q\in\Bbbk^{\times}$.  Let $\Bbbk_q[x_1,\cdots,x_n]$
denote the skew polynomial ring generated by $x_1,\cdots,x_n$ subject 
to the relations
\begin{equation}
\label{E0.0.1}\tag{E0.0.1}
x_j x_i =q x_i x_j, \quad \forall \; 1\leq i< j\leq n.
\end{equation}
In this paper we assume that $q$ is a nontrivial root of unity.
First we consider the case when $q=-1 (\neq 1)$.
By \cite[Theorem 4.10(1)]{CPWZ1}, if $n$ is even, then 
\begin{equation}
\label{E0.0.2}\tag{E0.0.2}
\Aut(\Bbbk_{-1}[x_1,\cdots,x_n])=S_n \ltimes (\Bbbk^{\times})^n
\end{equation}
which is virtually abelian [Definition \ref{xxdef0.6}(1)]. 
On the other hand, if $n$ is odd,
then \cite[Theorem 2]{CPWZ3} says that $\Aut(\Bbbk_{-1}[x_1,\cdots,x_n])$
contains a free group on two generators. In the case of $q=-1$, 
these two results present a dichotomy depending on the parity of $n$. 
This is a version of the Tits alternative [Definition \ref{xxdef0.6}(3)]. 

For any ${\mathbb Z}$-graded algebra $A=\bigoplus_{i\in {\mathbb Z}} A_i$ 
and for any positive integer $v$, the {\it  $v$th Veronese subring} of $A$ is 
defined to be
$$A^{(v)}:=\bigoplus_{i\in {\mathbb Z}} A_{vi}.$$
The following theorem generalizes the result \cite[Theorem 4.10(1)]{CPWZ1}.

\begin{theorem}
\label{xxthm0.1} 
Suppose that ${\rm{char}}\; \Bbbk\neq 2$.
Let $A$ be $\Bbbk_{-1}[x_1,\cdots,x_n]^{(v)}$ where $v$ is a 
positive integer. If $n$ and $v$ have different parity, then 
$\Aut(A)\cong S_n \ltimes (\Bbbk^{\times})^n$.
\end{theorem}

Considering elements in $(\Bbbk^{\times})^{n}$ as $(a_1,\cdots,a_n)$,
the $S_n$-action on $(\Bbbk^{\times})^n$ in Theorem \ref{xxthm0.1}
does not follow the standard rule 
\begin{equation}
\label{E0.1.1}\tag{E0.1.1}
\sigma: (a_1,\cdots,a_n)\longmapsto 
(a_{\sigma^{-1}(1)},\cdots,a_{\sigma^{-1}(n)})
\end{equation} 
for all $\sigma \in S_n$, due to asymmetry of the automorphisms corresponding to 
$(\Bbbk^{\times})^n$, see Lemma \ref{xxlem6.2} for some details. On the
other hand, the $S_n$-action appearing in \eqref{E0.0.2} does follow the 
standard rule \eqref{E0.1.1}.

If $n$ and $v$ have the same parity, we are unable to determine 
the automorphism group of $A$, but we conjecture that it contains
a free subgroup of rank $2$. Also in this case we are unable to 
decide whether or not $A$ is cancellative [Definition \ref{xxdef0.3}], 
-- see Theorem \ref{xxthm0.4} and Question \ref{xxque0.5} below 
for related results and questions.

We can generalize the above theorem to the case when $q$ is arbitrary of finite
order. Let $m$ be the order of $q$ and assume that $m$ is 
bigger than $2$ (or equivalently, $q\neq \pm 1$). We 
have two different hypotheses dependent on the parity 
of $n$ in the following theorem. 

\begin{theorem}
\label{xxthm0.2} 
Let $A$ be $\Bbbk_{q}[x_1,\cdots,x_n]^{(v)}$ where $v$ is a 
positive integer and $m>2$. Suppose that one of the following 
is true.
\begin{enumerate}
\item[(a)] 
$n$ is even and $m$ does not divide $v$.
\item[(b)]
$n$ is odd and $\gcd(m,v)\neq 1$. 
\end{enumerate}
Then the following statements hold.
\begin{enumerate}
\item[(1)]
If $q^v$ is either $1$ or $-1$, then  
$\Aut(A)\cong \Z/(n) \ltimes (\Bbbk^{\times})^n$.
\item[(2)]
If $q^v\neq \pm 1$, then $\Aut(A)\cong (\Bbbk^{\times})^n$.
\end{enumerate}
\end{theorem}

It is very difficult to describe the group $\Aut(A)$ if $n\geq 3$ and 
$(n,m,v)$ does not satisfy Theorem \ref{xxthm0.2}(a,b).
Hypotheses (a) and (b) have 
other significant consequences. Furthermore, for any tensor product 
of algebras in above two theorems, the automorphism 
group is also computable, see Remark \ref{xxrem7.8}(1). 


The proofs of the first two theorems are based on 
calculations of the discriminant of the algebra $A$ 
over its center. Further, the discriminant method 
can also be used to answer the {\it cancellation problem} 
which is closely related to the {\it automorphism problem}. 
We recall a definition.

\begin{definition}
\label{xxdef0.3}
An algebra $A$ is called {\it cancellative} if 
$A[t] \cong B[t]$ for 
any algebra $B$ implies that $A\cong B$. 
\end{definition}

One famous open problem in affine algebraic geometry is 
the Zariski Cancellation Problem which asks if \emph{the polynomial 
ring $\Bbbk[x_1, \cdots, x_n]$, for $n\geq 3$, is cancellative}. 
It is well-known that $\Bbbk[x]$ and $\Bbbk[x_1,x_2]$ are 
cancellative for any field $\Bbbk$. In 2013, Gupta \cite{Gu1, Gu2} settled
the Zariski Cancellation Problem negatively in positive characteristic 
for $n\geq 3$. The Zariski Cancellation Problem in characteristic
zero remains open for $n \geq 3$, see \cite{BZ, Gu3} for more 
details and relevant references.

Our methods of using the discriminant can be applied to show that certain 
Veronese subalgebras of the skew polynomial rings are cancellative.

\begin{theorem}
\label{xxthm0.4} 
Let $A$ be $\Bbbk_{q}[x_1,\cdots,x_n]^{(v)}$ where $v$ is a 
positive integer and let $m$ be the order of $q$.
Suppose that one of the following is true.
\begin{enumerate}
\item[(a)] 
$n$ is even and $m$ does not divide $v$.
\item[(b)]
$n$ is odd and $\gcd(m,v)\neq 1$. 
\end{enumerate}
Then $A$ is cancellative.
\end{theorem}

This says that all the algebras appearing in the first
two theorems are cancellative, see Remark \ref{xxrem7.8}(2) 
for a more general result. As mentioned above, we can not 
decide whether or not 
$\Bbbk_{q}[x_1,\cdots,x_n]^{(v)}$ is cancellative 
if it does not fit into Theorem \ref{xxthm0.4}.
We formally ask

\begin{question}
\label{xxque0.5}
Let $A$ be $\Bbbk_{q}[x_1,\cdots,x_n]^{(v)}$ where $v$ is a 
positive integer and let $2\leq m<\infty$ be the order of $q$.
Suppose that one of the following is true.
\begin{enumerate}
\item[(a)] 
$n$ is even and $m$ divides $v$.
\item[(b)]
$n$ is odd and $\gcd(m,v)= 1$. 
\end{enumerate}
Is then $A$ cancellative?
\end{question}


The Zariski Cancellation Problem is connected to several 
other open problems in affine algebraic geometry -- see \cite{BZ, Gu3}.  
In the noncommutative setting, it is also related to
certain properties of the Nakayama automorphism \cite{LWZ} and 
the Makar-Limanov invariant \cite{BZ}. 

The last result in this paper concerns the Tits alternative for 
automorphism groups of the Veronese subalgebras of skew 
polynomial rings. In 1972, Tits proved a remarkable 
and surprising dichotomy \cite{Ti}: \emph{for any subgroup $G$ of 
the general linear group ${\mathrm{GL}}({\mathbb C}^{\oplus n})$,
either $G$ is virtually solvable, or $G$ contains a
free group of rank $2$}. Since then, similar dichotomy
results have generally been referred as the {\it Tits alternative}.
The original Tits alternative and its variations 
have many applications in dynamical systems, geometric 
group theory, Diophantine geometry, topology and so on.
There is a version of the Tits alternative for the 
class of the automorphism groups of skew polynomial rings
following \cite[Theorem 2]{CPWZ3}. In general it would be 
very interesting to prove that some classes of algebraic objects 
must satisfy certain non-obvious dichotomy such as the Tits 
alternative. 

To state our result we  recall some definitions.

\begin{definition}
\label{xxdef0.6} Let $G$ be a group.
\begin{enumerate}
\item[(1)]
$G$ is called \emph{virtually abelian} if there is a normal
abelian subgroup $N\subseteq G$ such that $G/N$ is finite.
\item[(2)]
$G$ is called \emph{virtually solvable} if there is a normal
solvable subgroup $N\subseteq G$ such that $G/N$ is finite.
\item[(3)]
Let ${\mathcal C}$ be a class of groups. We say 
${\mathcal C}$ satisfies the {\it Tits Alternative} if 
the following dichotomy holds: any $G \in {\mathcal C}$ 
is either virtually solvable or it contains a free subgroup
of rank $2$.
\end{enumerate}
\end{definition}

For any fixed $n\geq 2$, let ${\mathcal C}_n$ consist of
groups $\Aut(A)$ where $A=\Bbbk_{q}[x_1,\cdots,x_n]^{(v)}$
for all $q\in\Bbbk^{\times}$ being a root of unity and all 
$v\in\mathbb{N}$. 

\begin{theorem}
\label{xxthm0.7}
Retain the above notation.
\begin{enumerate}
\item[(1)] 
If $n$ is odd, the Tits alternative holds for
${\mathcal C}_n$. 
\item[(2)]
The Tits alternative holds for ${\mathcal C}_2$. 
\end{enumerate}
\end{theorem}

This theorem leaves the following question.

\begin{question}
\label{xxthm0.8}
Does the Tits alternative hold for ${\mathcal C}_n$
for even integer $n\geq 4$?
\end{question}

In principle, the discriminant method introduced in \cite{CPWZ1, CPWZ2}
can be applied to any algebras, though in applications (and 
examples) given there  most algebras are Artin-Schelter 
regular. In this paper we consider a class of algebras 
that are not Artin-Schelter regular and show that the discriminant 
method is still very effective in solving several classical problems.

The paper is organized as follows. We provide background material 
and recall the definition of the discriminant in the noncommutative setting 
in Section 1. In Section 2, we study some basic properties of the 
discriminant. In Section 3, we provide some information about the center
and Veronese subrings of the $q$-skew polynomial rings. Detailed discriminant 
computations are given in Section 4 (when $n$ is odd) and Section 5 
(when $n$ is even). Main theorems (Theorems \ref{xxthm0.1} and  
\ref{xxthm0.2}) are proved in Section 6. In Section 
7 we deal with the cancellation problem and prove Theorem \ref{xxthm0.4}.
The Tits alternative is discussed in Section 8 where
Theorem \ref{xxthm0.7} is proved. 

\subsection*{Acknowledgments} 
A.A. Young was partly supported by the US
National Science Foundation (NSF Postdoctoral Research Fellowship,
No. DMS-1203744) and J.J. Zhang  by the US
National Science Foundation (No. DMS-1402863).

\section{Definitions}
\label{xxsec1}

Throughout the paper let $\Bbbk$ be a commutative domain, and sometimes we 
further assume that $\Bbbk$ is a field. Modules, vector spaces, algebras, 
tensor products, and morphisms are over $\Bbbk$. All algebras are associative 
with unit.

We will recall some definitions given in \cite{CPWZ1, CPWZ2} and introduce
some new definitions. In particular, we will introduce a new variant of the
discriminant in this section.

Let $B=M_w(R)$ be the $w\times w$-matrix algebra over a commutative domain 
$R$. We have the internal trace 
$$tr_{int}:B\to R, \quad (b_{ij})_{w\times w}\mapsto \sum_{i=1}^w b_{ii}$$
which is the usual matrix trace. Now let $B$ be a general $R$-algebra and
$F$ be a localization of $R$ such that that $B_F:=B\otimes_R F$ is 
finitely generated and free over $F$. Then the left multiplication defines a 
natural embedding of $R$-algebras 
\begin{equation}
\label{E1.0.1}\tag{E1.0.1}
lm: \quad B\to B_F\to \End_{F}(B_F)\cong M_{w}(F),
\end{equation}
where $w$ is the rank $rk_F(B_F)$. We define the {\it regular trace} map
by composing
\begin{equation}
\label{E1.0.2}\tag{E1.0.2}
tr_{reg}: \quad  B \xrightarrow{lm} M_w(F) \xrightarrow{tr_{int}} F
\subseteq Q(R)
\end{equation}
where $Q(R)$ is the field of fractions of $R$. 
Note that $tr_{reg}$ is independent of
the choices of $F$. 
In this paper, a trace (function) means the regular trace unless otherwise
stated. In computation, we also need to assume that the image of
$tr_{reg}$ is in $R$.

Let $R^{\times}$ denote the set of invertible elements in $R$. If 
$f, g\in R$ and $f=cg$ for some $c \in R^{\times}$, then we write 
$f=_{R^{\times}} g$. 

Let $A$ be a domain. We say a normal element $x\in A$ {\it divides} 
$y\in A$ if $y=xz$ for some $z\in A$. If $D$ is a 
set of elements in $A$, a normal element $x \in A$ is called a 
{\it common divisor} of $D$ if $x$ divides $d$ for all $d\in D$. 
We say a normal element $x\in A$ is the {\it greatest common divisor}
or $gcd$ of $D$, denoted by $\gcd_A D$, if
\begin{enumerate}
\item[(1)]
$x$ is a common divisor of $D$, and
\item[(2)]
any common divisor $y$ of $D$ divides $x$.
\end{enumerate}

It follows from part (2) that the gcd of any subset $D\subseteq A$ 
(if it exists) is unique up to a scalar in $A^\times$.

In practice, we often choose a domain $A$ such that 
$R\subseteq A \subseteq B$. Note that given $D\subseteq R$, 
the elements $\gcd_R D, \gcd_A D, \gcd_B D$ may not all exist. 
Even when they exist, they may not be equal.

\begin{definition}
\label{xxdef1.1}
Let $R$ be a commutative domain and $B$ be an $R$-algebra.
Suppose that the image of $tr:=tr_{reg}$ in \eqref{E1.0.2} is 
in $R$. Let $(r,p)$ be a pair of positive integers.
Let $A$ be a fixed domain between $R$ and $B$ in part (3).
\begin{enumerate}
\item[(1)]
\cite[Definition 1.2(1)]{CPWZ2}
Let $Z=\{z_i\}_{i=1}^r$ and $Z'=\{z'_i\}_{i=1}^r$ be two 
$r$-element subsets of $B$. The {\it discriminant} of the pair 
$(Z,Z')$ is defined to be
$$d_r(Z,Z')=\det(tr(z_iz'_j)_{r\times r})\in R.$$
\item[(2)]
The {\it $p$-power discriminant ideal of rank $r$}, denoted by 
$D_{r}^{[p]}(B/R)$, is the ideal of $R$ generated by the set
of elements of the form 
\begin{equation}
\label{E1.1.1}\tag{E1.1.1}
d_{r}(Z_1, Z_2)d_{r}(Z_3, Z_4)\cdots d_{r}(Z_{2p-1}, Z_{2p})
\end{equation}
for all possible $r$-element subsets $Z_1,Z_2, \cdots, Z_{2p}\subseteq B$.
\item[(3)]
The {\it $p$-power discriminant of rank $r$}, denoted by 
$d_{r}^{[p]}(B/R)$, is defined to be the gcd in $A$ of the elements 
of the form \eqref{E1.1.1}. Equivalently, the $p$-power discriminant 
$d_{r}^{[p]}(B/R)$ of rank $r$ is the gcd in $A$ of the elements in 
$D_{r}^{[p]}(B/R)$.
\end{enumerate}
\end{definition} 

The notation $d_{r}^{[p]}(B/R)$ suppresses the dependence of
the $p$-power discriminant of rank $r$ on the choice of an intermediate
domain $A$. In applications, this choice of $A$ will be clearly specified. 
Allowing  different choices of $A$, as well as different $p$ and $r$, 
increases the probability for the existence of the $\gcd$ of
elements in \eqref{E1.1.1}. When $p=1$, 
the above definition agrees with \cite[Definition 1.2]{CPWZ2}.
If $d_{r}^{[p]}(B/R)$ exists, then the ideal of $A$ generated by 
$d_{r}^{[p]}(B/R)$ is the smallest principal ideal of $A$ which is 
generated by a normal element and contains $D_{r}^{[p]}(B/R)$.

The following lemma is \cite[Proposition 1.4(3)]{CPWZ1}.

\begin{lemma}
\label{xxlem1.2}
Suppose $B$ is finitely generated and free over $R$ of
rank $r$. Then 
$$d^{[p]}_r(B/R)=_{A^{\times}} (d_{r}(B/R))^{p}$$
and $D^{[p]}_r(B/R)$ is the principal ideal of $R$ generated
by $(d_{r}(B/R))^{p}$.
\end{lemma}

\begin{proof} Let $X=\{x_1,\cdots,x_r\}$ be a basis 
of $B$ over $R$. By \cite[Definition 1.3(3)]{CPWZ1},
$d_r(B/R)=_{R^{\times}}d_r(X,X)$.
For any $r$-element subset
$Z:=\{z_1,\cdots,z_r\}\subset B$, we can write 
$z_i=\sum_{j} r_{ij} x_j$ for an $r\times r$-matrix
$(r_{ij})$. Similarly for $Z'$. Then 
$$d_r(Z,Z')=d_r(X,X) \det(r_{ij}) \det (r'_{ij})
=d_r(B/R)\det(r_{ij}) \det (r'_{ij}).$$
Then the assertion follows from the definition.
\end{proof}

\begin{lemma}
\label{xxlem1.3}
Let $\Psi$ be a subset of $B$ that generates $B$ as an $R$-module.
\begin{enumerate}
\item[(1)]
$D^{[p]}_r(B/R)$ is the ideal of $R$ generated by the set
\begin{equation}
\label{E1.3.1}\tag{E1.3.1}
\{ d_r( X_1,X_2)\cdots d_r(X_{2p-1},X_{2p}) \mid 
X_i\subseteq \Psi,\; \forall \; i\}.
\end{equation}
\item[(2)]
$d^{[p]}_r(B/R)$ is the $\gcd$ in $A$ of elements in set 
\eqref{E1.3.1}.
\end{enumerate}
\end{lemma}

\begin{proof} Every element $z\in B$ is an $R$-linear 
combination of $\phi_i\in \Psi$. By bilinearity of 
$tr(zz')$ and multi-linearity of $\det$, every $d_r(Z,Z')$ 
is an $R$-linear combination of $d_r(X,X')$ where $X,X'$ 
are $r$-element subsets of $\Psi$. Therefore every element 
of the form \eqref{E1.1.1} is an $R$-linear combination of
elements in \eqref{E1.3.1}. The assertions follow.
\end{proof}

In this paper we will see that some discriminants satisfy
the following.

\begin{definition}
\label{xxdef1.4} Retain the notation as in Definition
{\rm{\ref{xxdef1.1}}}. The $p$-power $r$-rank 
discriminant $d^{[p]}_r(B/R)$ is called {\it stable} 
if
$$d^{[ip]}_r(B/R)=_{A^{\times}} (d^{[p]}_r(B/R))^i$$
for all positive integers $i$.
\end{definition}

Under the hypotheses of Lemma \ref{xxlem1.2}, 
$d^{[p]}_r(B/R)$ is always stable for every $p$.

\section{Properties of the discriminant}
\label{xxsec2}

In this section we list of elementary properties of 
$d^{[p]}_r(B/R)$. The following lemma is similar to
\cite[Lemma 1.4]{CPWZ2}.

\begin{lemma}
\label{xxlem2.1} Suppose that the image of the regular trace
$\tr$ is in $R$. Let $g$ be an automorphism of $B$ such that
$g$ and $g^{-1}$ preserve $R$.
\begin{enumerate}
\item[(1)]
The $p$-power $r$-rank  discriminant ideal $D^{[p]}_r(B/R)$ 
is $g$-invariant. 
\item[(2)]
The $p$-power $r$-rank  discriminant $d^{[p]}_r(B/R)$ 
{\rm{(}}if exists{\rm{)}} is $g$-invariant up to a unit in $A$. 
\item[(3)]
Suppose $r_1\leq r_2$ and $p_1\leq p_2$ are positive integers. 
Then 
$$D^{[p_2]}_{r_2}(B/R) \subseteq D^{[p_1]}_{r_1}(B/R).$$
If both $d^{[p_2]}_{r_2}(B/R)$ and  $d^{[p_1]}_{r_1}(B/R)$ exist,
then 
$$d^{[p_1]}_{r_1}(B/R)\mid d^{[p_2]}_{r_2}(B/R).$$
As a consequence, the quotient 
$d^{[p_2]}_{r_2}(B/R)/d^{[p_1]}_{r_1}(B/R)$ is $g$-invariant 
up to a unit in $A$. 
\end{enumerate}
\end{lemma}

\begin{proof}
(1) By \cite[Lemma 1.8(3)]{CPWZ1}, $g(d_r(Z,Z'))=d_r(g(Z), g(Z'))$.
Then $g$ maps an element of the form \eqref{E1.1.1} to another
element of the same form. Similarly, this holds for $g^{-1}$.
Hence, $g$ (and $g^{-1}$) preserves $D^{[p]}_r(B/R)$.

(2) This follows from part (1) and the fact that the 
$\gcd$ is well-defined up to a unit. 

(3) When $p_1=p_2=1$, this is \cite[Lemma 1.4(5)]{CPWZ2}. For 
general $p_1\leq p_2$, the proof is similar to the proof of 
\cite[Lemma 1.4(5)]{CPWZ2}, so it is omitted.
\end{proof}

We recall some definitions from \cite[p.766]{CPWZ2}. Let $C$ be a 
domain such that $\Bbbk\subseteq C$ and that $C/\Bbbk$ is $\Bbbk$-flat. 
We say that $A\otimes C$ is $A$-closed if, for every 
$0\neq f\in A$ and $x,y\in A\otimes C$, the equation $xy=f$ implies
that $x,y\in A$ up to units of $A\otimes C$. For example, if $C$ is
connected graded and $A\otimes C$ is a domain, then $A\otimes C$ is
$A$-closed. The next lemma is similar to \cite[Lemma 1.12]{CPWZ2}.

\begin{lemma}
\label{xxlem2.2}
Retain the hypotheses as above. Assume that $B\otimes C$ 
is a domain.
\begin{enumerate}
\item[(1)]
$D^{[p]}_r(B\otimes C/R\otimes C)=D^{[p]}_r(B/R)\otimes C$.
\item[(2)]
Suppose $A\otimes C$ is $A$-closed. If $d^{[p]}_r(B/R)$ exists,
then $d^{[p]}_r(B\otimes C/R\otimes C)$ exists and
equals $d^{[p]}_r(B/R)$.
\end{enumerate}
\end{lemma}

\begin{proof} (1) First of all, the regular trace $\tr$ of $B\otimes C$
over $R\otimes C$ is equal to the regular trace $\tr$ of $B$
over $R$ when restricted to elements in $B$. 

Let $\Psi$ be a subset of $B$ such that 
$B$ is generated by $\Psi$ as an $R$-module. Then $B\otimes C$
is generated by $\Psi$ as an $R\otimes C$-module. By 
Lemma \ref{xxlem1.3}(1), $D^{[p]}_r(B\otimes C/R\otimes C)$ is the 
ideal of $R\otimes C$ generated by the set \eqref{E1.3.1}, 
which is just $D^{[p]}_r(B/R)\otimes C$ by Lemma \ref{xxlem1.3}(1).

(2) Suppose $d:=d^{[p]}_r(B/R)$ exists. Then it is the 
$\gcd$ in $A$ of the set $D^{[p]}_r(B/R)$ by definition.
By part (1), $d^{[p]}_r(B\otimes C/R\otimes C)$ (if exists)
is the $\gcd$ in $A\otimes C$ of the set 
$D^{[p]}_r(B/R)\otimes C$, which is the $\gcd$ in $A\otimes C$
of the set $D^{[p]}_r(B/R)$.

Let $d'\in A\otimes C$ be a common divisor in $A\otimes C$ of the set 
$D^{[p]}_r(B/R)$. By $A$-closedness of $A\otimes C$, we
may assume that $d'$ is in $A$ (up to a unit). This implies
that $d'$ divides $d$. It is clear that $d$ is a common
divisor in $A\otimes C$ of the set $D^{[p]}_r(B/R)$. Therefore
$d$ is the $\gcd$ in $A\otimes C$ of the set $D^{[p]}_r(B/R)$.
The assertion follows.
\end{proof}

Let $F$ be a localization of a commutative domain $R$ (and $R$ is
the center $Z(B)$ in most of applications) and 
$F$ may not be the fraction field of $R$. We assume that 
$B_F:=B\otimes_R F$ is finitely generated and free over $F$. 
We recall a definition from \cite{CPWZ2}.

\begin{definition}\cite[Definition 1.10]{CPWZ2}
\label{xxdef2.3}
Retain the above notations.
\begin{enumerate}
\item[(1)]
A subset $\bfb =\{b_1,\cdots, b_w\} \subseteq B$ is called a 
{\it semi-basis} of $B$ if it is an $F$-basis of $B_F$, 
where $b_i$ is viewed as $b_i\otimes_R 1 \in B_F$. 
In this case $w$ is the rank of $B$ over $R$. 
\item[(2)]
Let $\bfb$ be a semi-basis of $B$ and $T$ be a subset of $B$
containing $\bfb$ which generates $B$ as an $R$-module. We call such a set $T$ an
$R$-generating set of $B$. Then $\bfb$ is called a
{\em{quasi-basis}} (with respect to $T$) of $B$ if every $t \in T$ can be
written as $t=c b$ for some $b \in \bfb$. We denote $c$ by $( t:b )$.

%
\end{enumerate}
\end{definition}

We continue to introduce some notation. Again 
let $w$ be the rank of $B$ over $R$. Let 
$Z:=\{z_1, \cdots, z_w\}$ be a subset of $B$. If $\bfb$ 
is a semi-basis, then for each $i$,
$$z_i=\sum_{j=1}^w a_{ij} b_j$$
for some $a_{ij}\in F$.
In this case, the $w\times w$-matrix $(a_{ij})_{w\times w}$ 
is denoted by $(Z:\bfb)$. Let $T$ be as in Definition 
\ref{xxdef2.3}(2). Let $T/\bfb$ denote the subset of 
$F$ consisting of nonzero scalars of the form $\det(Z:\bfb)$ 
for all $Z\subseteq T$ with $|Z|=w$. Let
\begin{equation}
\label{E2.3.1}\tag{E2.3.1}
{\mathcal D}(T/\bfb) = \{d_w(\bfb,\bfb)f f'\mid
f,f'\in T/\bfb\},
\end{equation}
and
\begin{equation}
\label{E2.3.2}\tag{E2.3.2}
{\mathcal D}_w(T)=\{ d_w(Z,Z')\mid Z, Z' \subseteq T,\; |Z|=|Z'|=w\}.
\end{equation}
Note that if $Z$ and $Z'$ are $w$-element subsets of $T$, 
then $d_w(Z,Z')\in {\mathcal D}(T/\bfb)$ by \cite[(1.10.1)]{CPWZ2}.
In fact, we have ${\mathcal D}(T/\bfb)={\mathcal D}_w(T)$.

If $\bfb =\{b_1. \cdots, b_w\}$ is a quasi-basis with respect to an $R$-generating set $T$. Then for each $i$, let 
\begin{equation}
\label{E2.3.3}\tag{E2.3.3}
C_i = \{ (t:b_i) \mid t\in T\}\backslash\{0\}.
\end{equation}
It is 
easy to see that every element in $T/\bfb$ is of the form 
$c_1c_2\cdots c_w$, where $c_i\in C_i$ for each $i$. Let
\begin{equation}
\label{E2.3.4}\tag{E2.3.4}
{\mathcal D}^c(T/\bfb)=\left\{ d_w (\bfb,\bfb) \prod_{i=1}^w (c_i c_i')
\, \middle| \, c_i, c'_i\in C_i\right\}.
\end{equation}

If $\bfb$ is a quasi-basis with respect to $T$, 
then ${\mathcal D}(T/\bfb) ={\mathcal D}^c(T/\bfb)$.

Let $S$ be a subset of an algebra $R$. Let $S^p$ denote
the subset of $R$ consisting of $s_1s_2\cdots s_p$ 
for all $s_i\in S$.
The following lemma is similar to \cite[Lemma 1.11]{CPWZ2}.

\begin{lemma}
\label{xxlem2.4}
Let $T$ be a set of generators of $B$ as an $R$-module and 
$w={\mathrm{rank}}(B/R)$. Let $p$ be a positive integer.
\begin{enumerate}
\item[(1)]
$D^{[p]}_w(B/R)$ is generated by the set 
${\mathcal D}_w(T)^p$.
\item[(2)]
$d^{[p]}_w(B/R)=_{A^{\times}} \gcd {\mathcal D}_w(T)^p$.
\item[(3)]
If $\bfb$ is a semi-basis of $B$, then 
$d^{[p]}_w(B/R)=\gcd {\mathcal D}(T/\bfb)^p$.
\item[(4)]
If $\bfb$ is a quasi-basis of $B$, then 
$d^{[p]}_w(B/R)=\gcd ({\mathcal D}^c(T/\bfb))^p$.
\end{enumerate}
\end{lemma}

\begin{proof} (1) This follows from Lemma \ref{xxlem1.3}(1).

(2), (3) and (4) follow from the definition, the above 
discussion and part (1).
\end{proof}

In the rest of the section we assume that $B^1$ and $B^2$ are 
two algebras that are $\Bbbk$-flat. If $X^1\subseteq B^1$ 
and $X^2\subseteq B^2$ are two subsets, then $X^1\otimes X^2$ 
denotes the set $\{x\otimes y\mid x\in X^1, y\in X^2\}$.
We say that the pair $(X^1,X^2)$ is {\it hereditary} 
if  for $x\in X^1$ and $y\in X^2$, every divisor of $x\otimes y$
is of the form $x'\otimes y'$ (up to a unit in $B^1\otimes B^2$). 
The following lemma is easy.

\begin{lemma}
\label{xxlem2.5}
Let $X^1\subseteq B^1$ 
and $X^2\subseteq B^2$ be two subsets such that
\begin{enumerate}
\item[(a)] 
$\gcd X^1$ and $\gcd X^2$ exist.
\item[(b)]
$(X^1, X^2)$ is hereditary.
\end{enumerate} 
Then $\gcd (X^1\otimes X_2)=_{(B^1\otimes B^2)^{\times}}
\gcd X^1\otimes \gcd X^2$. \hfill\qed
\end{lemma}

\begin{lemma}
\label{xxlem2.6} 
Let $B^1$ and $B^2$ be two $\Bbbk$-algebras 
containing central subalgebra domains $R^1$ and $R^2$
respectively. Let $w_i=\mathrm{rank}_{R^i}(B^i)$ for $i=1,2$. 
Assume that
\begin{enumerate}
\item[(a)]
$R:=R^1\otimes R^2$ is a domain.
\item[(b)]
$\bfb^i$ is a quasi-basis of $B^i$ over $R^i$
with corresponding $R^i$-module generating set $T^i$
{\rm{(}}and $\bfb^i\subseteq T^i${\rm{)}}. 
\end{enumerate}
Then the following hold.
\begin{enumerate}
\item[(1)]
$\bfb:=\bfb^1\otimes \bfb^2$ is a quasi-basis 
of $B:=B^1\otimes B^2$ over $R$ 
with corresponding $R$-generating 
set being $T:=T^1\otimes T^2$.
\item[(2)]
Let $w:=w_1w_2$. Then ${\mathcal D}^c(T/\bfb)=
{\mathcal D}^c(T^1/\bfb^1)^{w_2}\otimes 
{\mathcal D}^c(T^2/\bfb^2)^{w_1}$.
\item[(3)]
Suppose that $({\mathcal D}^c(T^1/\bfb^1)^{w_2},
{\mathcal D}^c(T^2/\bfb^2)^{w_1})$ is hereditary.
Let $p$ be an integer. If $d^{[p]}_{w_1}(B^1/R^1)$
and $d^{[p]}_{w_2}(B^2/R^2)$ are stable discriminants, 
then so is $d^{[p]}_{w}(B/R)$. Further, 
$$d^{[p]}_{w}(B/R)=_{B^{\times}} d^{[p]}_{w_1}(B^1/R^1)^{w_2}
\otimes d^{[p]}_{w_2}(B^2/R^2)^{w_1}.$$
\end{enumerate}
\end{lemma}

\begin{proof}
(1) Let $i$ be either 1 or 2.
Let $\bfb^i=\{b^i_1, b^i_2, \cdots, b^i_{w_i}\}$
and $T^i=\{t^i_{j}\}_{j\in J^i}$. By definition, it 
is routine to check that $\bfb^1\otimes \bfb^2$ 
is a quasi-basis of $B^1\otimes B^2$ over $R$ 
with corresponding $R$-generating 
set being $T:=T^1\otimes T^2$.

(2) Let $1\leq j\leq w_1$ and $1\leq k \leq w_2$. 
Following \eqref{E2.3.3},
define $C^1_j$ be the set of elements $c$ in $B^1$
such that $cb^1_j\in T^1$. Similarly, we define
$C^2_k$. Let
$C_{j,k}$ consist of elements of the form
$c_j\otimes d_k$ where $c_j\in C^1_j$ and
$d_k\in C^2_k$. Then $C_{j,k}$ consist of
elements of the form 
$t (b^1_j\otimes b^2_{k})^{-1}$ for all 
$t\in T$. 

By linear algebra, $d_w(\bfb, \bfb)=d_{w_1}(\bfb^1, \bfb^1)^{w_2}
\otimes d_{w_2}(\bfb^2,\bfb^2)^{w_1}$. And
we have the following computation, for all
$(c_j \otimes d_k), (c'_j \otimes d'_k)\in C_{j,k}$ for different
$(j,k)$,
$$\prod_{j,k} [(c_j \otimes d_k)(c'_j \otimes d'_k)]
\in X^{w_2} \otimes Y^{w_1}$$
where $X=\{ (\prod_{j} (c_j c'_j))\mid c_j, c'_j\in C_j^1\}$ and
$Y=\{(\prod_{k} (d_k d'_k)) \mid d_k, d'_k\in C^2_k\}$. 
Now assertion follows from \eqref{E2.3.4}.

(3) This follows from the definition, Lemma \ref{xxlem2.4}(4),
part (2) and Lemma \ref{xxlem2.5}.
\end{proof}

\section{Center and Veronese subrings of $q$-polynomial rings}
\label{xxsec3}

From now on we fix two integers $m,n\geq 2$ and 
a primitive $m$th root of unity, say $q$, in $\Bbbk$. The 
$q$-skew polynomial ring is generated by $x_1,\cdots,x_n$
and subject to the relations
\begin{equation}
\label{E3.0.1}\tag{E3.0.1}
x_j x_i= q x_i x_j, \quad \forall\, 1\leq i<j \leq n.
\end{equation}
and is denoted by $\Bbbk_{q}[x_1,\cdots,x_n]$,
or simply by $\Bbbk_{q}[\bfx]$.

We will adopt the following notation for monomials
${\bfx}^{\bfs} \assign x_{1}^{s_1} \cdots x_{n}^{s_n}$ where
${\bfs} = ( s_1 , ... , s_n ) \in {\mathbb N}^n$ is its degree vector. 
We will also denote by $\bfe_i$ the standard basis vector, 
with $1$ in its $i$th component and $0$ elsewhere.
For any $0 \le k \le m$, define
\begin{equation}
\label{E3.0.2}\tag{E3.0.2}
y_k \assign q^{- \lfloor n/2 \rfloor k(k+1)/2} \;
{\bfx}^{(k, m-k, k, m-k, ...)},
\end{equation}
in particular,
$$
\begin{array}{ccc}
y_0 &=& x_2^m x_4^m \cdots x_{2 \lfloor n/2 \rfloor}^m, \\
y_m &=& (-1)^{\lfloor n/2 \rfloor(m+1)} 
x_1^m x_3^m \cdots x_{2 \lceil n/2 \rceil - 1}^m.
\end{array}
$$
Note that both $y_0$ and $y_m$ are in the central subalgebra 
generated by $\{x_1^m, ..., x_n^m\}$.
One can easily check that the $y_i$s satisfy the following 
relations
\begin{equation}
\label{E3.0.3}\tag{E3.0.3}
y_i y_j = q^{- \lfloor n/2 \rfloor (i+j)(i+j+1)/2} 
\bfx^{(i+j, 2m-i-j, i+j, 2m-i-j, ...)}, \quad
\forall \; 0 \le i, j \le m.
\end{equation}
As a consequence, 
\begin{equation}
\label{E3.0.4}\tag{E3.0.4}
y_i y_j = y_k y_\ell, \quad \forall\; i + j = k + \ell.
\end{equation} 
Equations \eqref{E3.0.3}-\eqref{E3.0.4} also imply that
\begin{equation}
\label{E3.0.5}\tag{E3.0.5}
y_i y_j = 
\left\{
\begin{array}{lr}
y_0 y_{i+j} & i + j \leq  m, \\
y_m y_{i+j-m} & i + j > m.
\end{array}
\right.
\end{equation}

The following is a consequence of \cite[Lemma 4.1]{CYZ}.
Let $Z(A)$ denote the center of an algebra $A$.

\begin{lemma}
\label{xxlem3.1}
\begin{enumerate}
\item[(1)] If $n$ is even, then $Z(\Bbbk_{q} [\bfx])$ 
is a polynomial ring generated by $x_1^m, ..., x_n^m$. 
\item[(2)]
If $n$ is odd, then $Z(\Bbbk_{q} [\bfx])$ 
is generated by $x_1^m, ..., x_n^m, y_1, ..., y_{m-1}$. 
\end{enumerate}
\end{lemma}

\begin{proof} (1) This is \cite[Example 2.4(2)]{CPWZ2}.

(2) One can check it directly or use \cite[Lemma 4.1]{CYZ}.
We use some of the notation in \cite[Section 4]{CYZ}. 
Let $Y$ be the skew symmetric $n\times n$-matrix with $1/m$ 
in all entries above the diagonal. 

Let $\mathbf{t} \in {\mathbb N}^n$. By \cite[Lemma 4.1]{CYZ} 
the monomial $\mathbf{x}^{\mathbf{t}}$ is in the center
$Z ( \Bbbk_{q} [ \mathbf{x} ] )$ if and only if $Y\mathbf{t} 
\in \Z^n$. Let $S=mY$. Then $\mathbf{x}^{\mathbf{t}}
\in Z ( \Bbbk_{q} [ \mathbf{x} ] )$ if and only if $S\mathbf{t} 
\in m\Z^n$. Let $\bar{S}$ be the endomorphism of 
$( \Z/m\Z )^n$ represented by the matrix $S$. Then 
$S \mathbf{t} \in m\Z^n$ if and only if $\mathbf{t}$ is a 
lift of an element in $\ker ( \bar{S} )$.
  
Since $n$ is odd, $\mathrm{rank}(\bar{S}\otimes\F_{p})
=n-1$ for all primes $p$. It is easy to check that 
$\ker ( \bar{S} )$ is generated by $(i,-i,i, \ldots ,-i,i ) 
\in ( \Z/m\Z )^{m}$ for $i=0, \ldots ,m$. Lifting these to 
$\Z^{n}$ gives $( i,m-i,i, \ldots ,m-i,i )$ for 
$i=1, \ldots ,m-1$ and $m \mathbf{e}_{i}$ for $i=1,\ldots,n$. 
\end{proof}

When $n$ is even, the center $Z(\Bbbk_q[\bfx])$ is easy to
understand, namely
$$Z(\Bbbk_q[\bfx])=\Bbbk[x^m_1, ... , x^m_n].$$ 
If $n$ is odd, every element of $Z(\Bbbk_q[\bfx])$ 
can be expressed as a linear combination of terms of the 
form ${\bfx}^{m \bfa}$ or ${\bfx}^{m {\bfa}} y_b$, 
with ${\bfa} \in \N^n$ and $0 < b < m$. 
Each such term can be rewritten as follows,
$${\bfx}^{m \bfa} y_b = x_1^{a_1m} \cdots x_n^{a_nm} y_b 
= q^{-\lfloor n/2 \rfloor b(b+1)/2} 
x_1^{a_1m + b} x_2^{(a_2 + 1)m - b} x_3^{a_3m+b} \cdots x_n^{a_nm+b}.$$
Since the above polynomials form a $\Bbbk$-linear basis
of $Z(\Bbbk_q[\bfx])$, we have
%
$$\begin{aligned}
Z(\Bbbk_q[{\bfx}]) &\cong \frac{\Bbbk[x_1^m, ..., x_n^m, y_0, ..., y_m]}
{\left(
\begin{array}{c}
y_0 - x_2^m x_4^m \cdots x_{n-1}^m, \\
y_m - (-1)^{(m+1)(n-1)/2} x_1^m x_3^m \cdots x_n^m, \\
y_i y_j-y_k y_{\ell}, \quad \forall\; i+j=k+\ell.
\end{array}
\right)} \\
&
\cong \frac{\Bbbk[x_1^m, ..., x_n^m, y_1, ..., y_{m-1}]}
{\left(
\begin{array}{c}
\forall i + j < m,\quad 
        y_i y_j - x_2^m x_4^m \cdots x_{n-1}^m y_{i+j}, \\
\forall i + j = m,\quad 
        y_i y_j - (-1)^{(m+1)(n-1)/2} x_1^m x_2^m \cdots x_n^m, \\
\forall i + j > m,\quad 
y_i y_j - (-1)^{(m+1)(n-1)/2} x_1^m x_3^m \cdots x_n^m y_{i+j-m}.
\end{array}
\right)}.
\end{aligned}
$$
For example, if $n = 3$,
$$Z(\Bbbk_q[{\bfx}]) \cong \frac{\Bbbk[x_1^m, x_2^m, x_3^m, y_0, ..., y_m]}{
\left(
\begin{array}{c}
y_0 - x_2^m, \quad y_m - (-1)^{m+1} x_1^m x_3^m, \\
y_i y_j - y_k y_{\ell}, \quad \forall i + j = k + \ell.
\end{array}
\right)},$$
and if $m = 2$,
$$Z(\Bbbk_q[{\bfx}]) \cong \frac{\Bbbk[x_1^2, ..., x_n^2, y_1]}{
\left(
\begin{array}{c}
y_1^2 - (-1)^{(n-1)/2} x_1^2 x_2^2 \cdots x_n^2.
\end{array}
\right)}.$$
Hopefully this gives some idea on what the center should be.

For any $v\in \N$, the $v$th Veronese subalgebra
of $\Bbbk_{q}[\bfx]$, denoted by  $\Bbbk_q [\bfx]^{(v)}$,
is the subalgebra generated by elements of total degree 
$v$. 

As before we fix positive integers $m,n,v$. 
Let 
\begin{equation}
\label{E3.1.1}\tag{E3.1.1}
g \assign \gcd (v,m).
\end{equation}

Let $\Bbbk_q[\bfx^{\pm 1}]$ be the localization of
$\Bbbk_q[\bfx]$ by inverting all $x_i$s. 
We extend the notation ${\bfx}^{\bf s}$ for all ${\bf s} \in \Z^n$
in a natural way.

Let $F$ be the center of $\Bbbk_q[{\bfx}^{\pm 1}]^{(v)}$
which is a localization of $Z := Z(\Bbbk_q[{\bfx}]^{(v)})$, 
and let $\Bbbk_q[{\bfx}]^{(v)}_F = \Bbbk_q[{\bfx}]^{(v)} \otimes_Z F$. 
Since $F$ is a $\Z^n$-graded field, we have
\begin{enumerate}
\item[(1)]
$\Bbbk_q[{\bfx}]^{(v)}_F=\Bbbk_q[{\bfx}^{\pm 1}]^{(v)}$ which is
a $\Z^n$-graded skew field. 
\item[(2)]
$\Bbbk_q[{\bfx}]^{(v)}_F$ is free over $F$.
\end{enumerate}
Since each $x_i^{mv} \in F$, we have that $\Bbbk_q[{\bfx}]^{(v)}_F$ is 
finite dimensional over $F$, and we denote 
\begin{equation}
\label{E3.1.2}\tag{E3.1.2}
w \assign \dim_F \Bbbk_q[{\bfx}]^{(v)}_F.
\end{equation}

Let $H_v = \{{\bf{s}} \in \Z^n \mid \sum_{i=1}^n s_i \in v \Z \}$, 
and let $H_v^{+} = H_v \cap \N^n$, so that $\Bbbk_q[{\bfx}]^{(v)}$ 
(respectively,  $\Bbbk_q[{\bfx}^{\pm 1}]^{(v)}$)
is the span of ${\bfx}^{H_v^{+}}$ (respectively,
${\bfx}^{H_v}$).

\begin{lemma}\label{xxlem3.2} 
Retain the above notation.
Suppose that $n$ is odd.
\begin{enumerate}
\item[(1)]
$$Z (\Bbbk_q [\bfx]^{(v)}) =
 Z(\Bbbk_q [\bfx]) \cap \Bbbk_q [\bfx]^{(v)} 
= \Bbbk \langle x_i^m, y_j \rangle \cap \Bbbk_q [\bfx]^{(v)}.$$
\item[(2)]
The center $Z(\Bbbk_q[\bfx^{\pm 1}]^{(v)})$ is spanned by
$\bfx^{M}$ where
$$M = \left( m\Z^n + g\Z \sum_{i = 1}^n (- 1)^{i - 1} {\bf e}_i \right) \cap H_v.$$
As a consequence,
$$Z (\Bbbk_q [\bfx]^{(v)}) 
= \Bbbk \langle x_i^m, y_{jg} \rangle \cap \Bbbk_q [\bfx]^{(v)}.$$
\end{enumerate}
\end{lemma}

\begin{proof}  
Let $\bfx^{\bf s}\in  Z (\Bbbk_q [\bfx^{\pm 1}]^{(v)})$ 
for some ${\bf s} \in \Z^n$. Since $x_i x_{i+1}^{vm-1}\in 
\Bbbk_q [\bfx]^{(v)}$,
$$\bfx^{\bf s} x_i x_{i+1}^{mv-1}
=x_i x_{i+1}^{mv-1} \bfx^{\bf s} 
= q^{-(s_i + s_{i+1})}  \bfx^{\bf s} x_i x_{i+1}^{mv-1}.$$
Hence, $s_i + s_{i+1} \in m \Z$ for all $i$. Then, 
for each $i$,
\begin{equation}
\label{E3.2.1}\tag{E3.2.1}
s_i = \begin{cases} a_i m +b & {\text{$i$ is odd,}}\\
a_im +(m-b) &{\text{$i$ is even}}.\end{cases}
\end{equation}
for some $a_1, ..., a_n \in \Z$ and $0\leq b\leq m-1$. 
This part of the proof works for both even and odd $n$.

(1) When $\bfx^{\bfs}\in Z (\Bbbk_q [\bfx]^{(v)})$ 
for ${\bf s} \in \N^n$. We obtain that, if $b>0$, then 
$a_i\geq 0$  for all $a_i$ in \eqref{E3.2.1} and if
$b=0$, $a_i\geq 0$ for odd $i$ and $a_i\geq -1$ for 
even $i$. This is equivalent to  
$$\bfx^{\bf s} =_{\Bbbk^\times} \begin{cases}
x_1^{ma_1}\cdots x_n^{ma_n} & b=0,\\
x_1^{ma_1} \cdots x_n^{ma_n} y_b & b\neq 0,
\end{cases}
$$
for some $a_i\geq 0$. The assertion follows.

(2) Recall that $n$ is odd.
Note that, if $x_1^{ma_1 } \cdots x_n^{ma_n} y_b 
\in Z(\Bbbk_q [\bfx^{\pm 1}]^{(v)})$, then 
$$b + m\left(\frac{n-1}{2} + \sum_{i=1}^n a_i\right) \in v \Z,$$ 
and hence, $b \in g \Z$. This means that if $\bfx^{\bfs}
\in Z$, then $\bfs\in M$. Conversely, it is straightforward
to check that if $\bfs\in M$, then $\bfx^{\bfs}\in Z$.
\end{proof}

We are interested the discriminant of $\Bbbk_q[{\bfx}]^{(v)}$
over its center. We examine separately the case when $n$ is 
odd, and the case when $n$ is even. 

We conclude this section with the hereditary property 
(as mentioned before Lemma \ref{xxlem2.5}) for 
monomials in $\Bbbk_{q}[\bfx]^{(v)}$.

\begin{lemma}
\label{xxlem3.3}
Let $A^1, \cdots, A^s$ be algebras of type 
$\Bbbk_{q}[\bfx]^{(v)}$. For each $i$, let
$X^i\subseteq A^i$ be a set of monomials. 
Then, for any $f^i\in X^i$, every divisor of
$f^1\otimes f^2\otimes \cdots \otimes f^s$ 
is of the form $g^1\otimes g^2\otimes \cdots
\otimes g^s$ where each $g^i$ is a divisor of $f^i$.
\end{lemma}

\begin{proof} Consider $A^i=\Bbbk_{q_i}[\bfx]^{(v_i)}$
as an ${\mathbb N}^{n_i}$-graded algebra for all
$i$. Let $n=\sum_{i=1}^s n_i$. Then $A^1\otimes
\cdots \otimes A^s$ is an ${\mathbb N}^{n}$-graded algebra. 
Since each $f^i$ is ${\mathbb N}^{n_i}$-homogeneous, 
$F:=f^1\otimes  \cdots \otimes f^s$ 
is ${\mathbb N}^{n}$-homogeneous. 
Note that ${\mathbb N}^{n}$ is an ordered 
semigroup. Then any divisor $G$ of $F$ is 
${\mathbb N}^{n}$-homogeneous.
Equivalently, $G=g^1\otimes \cdots \otimes
g^s$ where each $g^i$ is a divisor of $f^i$.
\end{proof}

\section{Discriminant computation: when $n$ is odd}
\label{xxsec4}

We will freely use the notation introduced in the last 
section, and further assume that $n$ is odd in a large 
part of this section.

Recall from Lemma \ref{xxlem3.2} that, if $n$ is odd, then
\begin{equation}
\label{E4.0.1}\tag{E4.0.1}
M = \left( m\Z^n + g\Z(\sum_{i = 1}^n (- 1)^{i - 1} {\bf e}_i) \right) \cap H_v.
\end{equation}
Then $M$ is a subgroup of $H_v$. 
We can partition $H_v$ into cosets mod $M$. It is easy to see 
the total number of these cosets is equal to $w$ \eqref{E3.1.2}.

\begin{lemma}
\label{xxlem4.1}
Assume $n$ is odd.
\begin{enumerate}
\item[(1)] 
For each coset of $M$ in $H_v$, there 
is a unique representative ${\bf p}:= (p_1, ..., p_n)$ such that
\begin{enumerate}
\item[(a)] 
$0 \le p_1 < g$,
\item[(b)] 
for each $1 < i < n$, we have $0 \le p_i < m$, and
\item[(c)] 
$0 \le p_n < vm/g$.
\end{enumerate}
Moreover, the above remains true with indices $(1,n)$ replaced by any $(\mu,\nu)$ with $\mu\ne\nu$. 
\item[(2)]
$w = m^{n-1}$.
\item[(3)] 
$w\neq 0$ in $\Bbbk$.
\end{enumerate} 
\end{lemma}

\begin{proof}
(1) Pick an arbitrary coset $M'$ of $M$, and let ${\bf p} = 
(p_1, ..., p_n) \in M'$. Since $g=\gcd(m,v)$, 
there exists $c \in \Z$ such that 
$cm \equiv g$ mod $v$. Hence  
$(g, -g, g, -g, ..., -g, g - cm) \in M$, and we can translate ${\bf p}$ by some multiple of 
this vector to obtain $0\le p_1<g$. 
Furthermore, if ${\bf t} \in M$ then 
$t_1 \in g \Z$, so there is no vector in $M'$ whose 
first component is any other $0 \le r' < g$. 

For each $1 < i < n$, we have $m ({\bf e}_i - {\bf e}_n) \in M$, 
so we can apply the translation trick above and assume that
$0 \le p_i < m$. Furthermore, if ${\bf t} \in M$ and 
$t_1 = 0$, then each other $t_i \in m\Z$. This implies
that there is no other set of possible values of 
$p_1, ..., p_{n-1}$ subject to the conditions 
$0 \le p_1 < g$ and $0 \le p_i < m$ for every $1<i<n$. 

Finally, $(vm/g) {\bf e}_n \in M$, so there exists a 
representative ${\bf p} \in M'$ subject to 
constraints (a)-(c) of the lemma. If $c {\bf e}_n \in M$, 
then $c \in m \Z \cap v \Z = (vm/g) \Z$, so this 
representative is unique. 

The last statement is clear since the above calculations 
do not depend on the ordering of the indices $1,...n$. 
This finishes the proof of part (1).

(2) The value $w$ can be determined by counting the cosets 
by their representatives. For every sequence of integers 
$p_1, ..., p_{n-1}$ such that $0 \le p_1 < g$ and 
$0 \le p_i < m$ for all $1<i<n$, there are $m/g$ 
possible values of $p_n$ such that $0 \le p_n < vm/g$ 
and $(p_1, ..., p_n) \in H_v$. Therefore, 
$w = g \cdot m^{n-2} \cdot m/g = m^{n-1}$. 

(3) Since $q\in \Bbbk$ and $o(q)=m$, the characteristic
of $\Bbbk$ cannot divide $m$. Or $m\neq 0$ and $w\neq 0$
in $\Bbbk$.
\end{proof}

In this paper we mainly consider the case when $B=\Bbbk_q[\bfx]^{(v)}$
for both even and odd $n$. 
Using \cite[Definition 1.10]{CPWZ1} and notation in Section 2, if 
$\mathcal{B} := \{{\bf b}_1, ..., {\bf b}_w\} \subseteq H_v^{+}$ is a 
set of representatives of each coset of $M$, then $\bfb:= \bfx^{\mathcal{B}}$ 
is a quasi-basis of $\Bbbk_q[\bfx]^{(v)}$ with respect to $T:=\bfx^{H_v^{+}}$, and
\begin{equation}
\label{E4.1.1}\tag{E4.1.1}
\mathcal{D}^c (T/\bfb) =_{\Bbbk^{\times}} d_w(\bfb,\bfb) 
\left\{ \prod_{i=1}^w \bfx^{{\bf s}_i - {\bf b}_i} \bfx^{{\bf s}'_i - {\bf b}_i} 
\, \middle| \, {\bf s}_i, {\bf s}'_i \in \N^n \cap ( M + {\bf b}_i ) \right\}.
\end{equation}

The following lemmas hold for both even and odd $n$.

\begin{lemma}
\label{xxlem4.2} 
Retain the above notation. Let $\bfs\in H_v^{+}$ such that
$\bfx^\bfs$ is not central. Then $\tr(\bfx^{\bfs})=0$.
As a consequence, the trace map $\tr$ sends $\Bbbk_{q}[\bfx]^{(v)}$
to $Z(\Bbbk_{q}[\bfx]^{(v)})$.
\end{lemma}

\begin{proof} Since $A:=\Bbbk_q[\bfx]^{(v)}$ is ${\mathbb Z}^n$-graded, so is
the center $Z:=Z(A)$. Let $F$ be the graded field of fractions 
of $Z$. Then $A$ is a free module over $F$ with $F$-basis $\mathcal{B}$.
Then, for all $i,j$, there is a unique $k$ such that $\bfb_i \bfb_j=c_{ij}^k \bfb_k$ 
for some $0\neq c_{ij}^k\in F$. If $\bfb_i$ is not in the center, 
then $j\neq k$. Therefore $\tr(\bfb_i)=\sum_{j=k} c_{ij}^k=0$. Every 
element $\bfx^{\bfs}$ is of the form $c \bfb_i$ for some $i$ and $c\in F$.
The assertion follows.
\end{proof}

\begin{lemma}
\label{xxlem4.3}
Retain the above notation. Suppose 
that $w$ is invertible. Then
$$\mathcal{D}^c(T/\bfb) =_{\Bbbk^\times} 
\left\{ \left( \prod_{i=1}^w \bfx^{ {\bf s}_i} \right)
\left( \prod_{i=1}^w \bfx^{ {\bf s}'_i} \right)
\, \middle| \, {\bf s}_i, {\bf s}'_i\in \N^n \cap ( M + {\bf b}_i )\right\}.
$$
\end{lemma}

\begin{proof}
For each ${\bf b}_i \in \mathcal{B}$, let ${\bf b}_i^* \in \mathcal{B}$ 
be such that ${\bf b}_i + {\bf b}_i^* \in M$. For any ${\bf s} \in H_v^{+}$, 
if ${\bf s} \notin M$, then $\bfx^{\bf s}$ is not central, and 
$\tr(\bfx^{\bf s}) = 0$ by Lemma \ref{xxlem4.2}.
If ${\bf s} \in M$, then $\bfx^{\bf s}$ is central, and $\tr(\bfx^{\bf s}) 
= w\bfx^{\bf s}=_{\Bbbk^{\times}} \bfx^{\bf s}$, where the
last equation follows from the hypothesis that $w$ is invertible. 
Therefore, in the matrix $(\tr(\bfx^{{\bf b}_i} \bfx^{{\bf b}_j}))_{w \times w}$, 
the only nonzero terms appear where ${\bf b}_j = {\bf b}_i^*$, and
$$d_w (\bfb,\bfb) 
=_{\Bbbk^\times} \det (\tr(\bfx^{{\bf b}_i} \bfx^{{\bf b}_j}))_{w \times w} 
=_{\Bbbk^\times} \prod_{i=1}^w \bfx^{{\bf b}_i} \bfx^{{\bf b}_i^*} 
=_{\Bbbk^\times} \left( \prod_{i=1}^w \bfx^{{\bf b}_i} \right)^2.$$
The assertion follows by the above formula and equation \eqref{E4.1.1}.
\end{proof}

Recall that $m$ is the order of $q$, the rank of $\Bbbk_q[\bfx]^{(v)}$ 
over its center is $w=m^{n-1}$ and $g=\gcd(v,m)$. 

\begin{theorem}
\label{xxthm4.4}
Let $B=\Bbbk_q[\bfx]^{(v)}$ when $n$ is odd. Suppose that 
$m$ is invertible in $\Bbbk$. Let $R$ be the
center of $B$. Assume that $v$ divides $wp(g-1)$. Then
$$d_w^{[p]}(B/R) =_{\Bbbk^{\times}} (x_1 x_2 \cdots x_n)^{wp(g-1)}
=_{\Bbbk^{\times}} (x_1^v x_2^v \cdots x_n^v)^{\frac{wp(g-1)}{v}}.$$
As a consequence, $d_w^{[p]}(B/R)$ is stable.
\end{theorem}

\begin{proof} By Lemmas \ref{xxlem2.4}(4)
and \ref{xxlem4.3}, we have $d_w^{[p]}(B/R) =\gcd \Lambda^{2p}$
where 
$$\Lambda:=\left\{ \prod_{i=1}^w \bfx^{ {\bf s}_i} 
\, \middle| \, {\bf s}_i\in \N^n \cap ( M + {\bf b}_i )\right\}.
$$
For each $1 \le s \le n$, let $f_s \in \N$ be maximal such 
that $x_s^{f_s}$ divides all elements of $\Lambda$. 
This gives $\bfx^{2p{\bf f}}$ as the gcd of $\Lambda^{2p}$ 
in the over-algebra $\Bbbk_q[\bfx]\supseteq B$ where ${\bf f}=
(f_1,\cdots,f_n)$. If $2p {\bf f} \in H_v^{+}$, then it is the gcd 
of $\Lambda^{2p}$ in $\Bbbk_q[\bfx]^{(v)}$ as well, 
but, otherwise, this is not true. 

We first calculate $f_1$ by summing the lowest 
powers of $x_1$ in each coset of $M$ (or more precisely, 
in each $\N^n \cap ( M + {\bf b}_i )$ for different 
$i$). These lowest powers can be found by using the 
representatives outlined in Lemma \ref{xxlem4.1}(1),
which also shows that this power cannot 
exceed $g - 1$. For each $0 \le k \le g-1$, 
there are $m^{n-1}/g$ cosets with 
lowest power $x_1^k$. Therefore, the sum is 
$$f_1 = \frac{m^{n-1}}{g} \; \frac{g(g-1)}{2} = \frac{w(g-1)}{2}.$$ 
For $f_i$ with $i\ne 1$, we can use the last assertion of
Lemma \ref{xxlem4.1}(1) to relabel indices, so the above calculation
remains valid for $i\ne 1$ and we conclude that $f_1 = f_2 = \cdots = f_n$.

Now $2p{\bf f}=wp(g-1) (1,1,\cdots,1)\in H_v^{+}$ as $v$ divides $wp(g-1)$. 
The assertion follows from the last paragraph, and stability of
$d_w^{[p]}(B/R)$ follows from the main assertion.
\end{proof}

\section{Discriminant computation: when $n$ is even}
\label{xxsec5}

In this section we assume that $n$ is even. The following
is similar to Lemma \ref{xxlem3.2}.

\begin{lemma}
\label{xxlem5.1} 
Suppose that $n$ is even.
\begin{enumerate}
\item[(1)]
$$Z(\Bbbk_q [\bfx]^{(v)}) = 
\Bbbk \langle x_i^m, y_{jm/g} \rangle \cap \Bbbk_q [\bfx]^{(v)}.$$
\item[(2)]
The center $Z(\Bbbk_q [\bfx^{\pm 1}]^{(v)})$ is spanned by $\bfx^M$ 
where
$$M:=\left( m\Z^n + \left(\frac{m}{g}\right)\Z \sum_{i = 1}^n (- 1)^{i - 1} {\bf e}_i \right) \cap H_v.$$
\end{enumerate}
\end{lemma}
\begin{proof}
(1) We copy the first part of the proof of 
Lemma \ref{xxlem3.2}.

Let $\bfx^{\bf s}\in  Z (\Bbbk_q [\bfx^{\pm 1}]^{(v)})$ 
for some ${\bf s} \in \Z^n$. Since $x_i x_{i+1}^{vm-1}\in 
\Bbbk_q [\bfx]^{(v)}$, we have
$$\bfx^{\bf s} x_i x_{i+1}^{mv-1}
=x_i x_{i+1}^{mv-1} \bfx^{\bf s} 
= q^{-(s_i + s_{i+1})}  \bfx^{\bf s} x_i x_{i+1}^{mv-1}.$$
Hence, $s_i + s_{i+1} \in m \Z$ for all $i$. Then, 
for each $i$,
\begin{equation}
\label{E5.1.1}\tag{E5.1.1}
s_i = \begin{cases} a_i m +b & {\text{$i$ is odd,}}\\
a_im +(m-b) &{\text{$i$ is even}}.\end{cases}
\end{equation}
for some $a_1, ..., a_n \in \Z$ and $0\leq b\leq m-1$. 
Considering $\bfx^{\bfs}\in Z (\Bbbk_q [\bfx]^{(v)})$ 
for ${\bf s} \in \N^n$. We obtain that, if $b>0$, then 
$a_i\geq 0$  for all $a_i$ in \eqref{E5.1.1} and if
$b=0$, $a_i\geq 0$ for odd $i$ and $a_i\geq -1$ for 
even $i$. This is equivalent to 
$$\bfx^{\bf s} =_{\Bbbk^\times} \begin{cases}
x_1^{ma_1}\cdots x_n^{ma_n} & b=0,\\
x_1^{ma_1} \cdots x_n^{ma_n} y_b & b\neq 0,
\end{cases}
$$
for some $a_i\geq 0$. Next we need to determine the values of
$b$ such that $y_b\in Z (\Bbbk_q [\bfx]^{(v)})$.
Note that $x_1^v y_b= q^{vb} y_b x_1^v$. Hence
$vb\in m\Z$, or equivalently, $b$ is a multiple of
$m/g$. The assertion follows.

(2) By the proof of part (1), every monomial 
$\bfx^{\bfs}\in Z(\Bbbk_q [\bfx^{\pm 1}]^{(v)})$ is 
in $\bfx^M$. Conversely, it is straightforward to
check that every element in $\bfx^M$ is also in 
$Z(\Bbbk_q [\bfx^{\pm 1}]^{(v)})$
\end{proof}

Much of the work of last section can be reapplied. When 
$n$ is even we define $M$ as in Lemma \ref{xxlem5.1}(2):

\begin{equation}
\label{E5.1.2}\tag{E5.1.2}
M=\left( m\Z^n + \left(\frac{m}{g}\right)\Z 
(\sum_{i = 1}^n (- 1)^{i - 1} {\bf e}_i) \right) \cap H_v.
\end{equation}

\begin{lemma} 
\label{xxlem5.2} Suppose that $n$ is even.
\begin{enumerate}
\item[(1)]
For each coset of $M$ in $H_v$, there is a unique 
representative ${\bf p}:=(p_1,\cdots,p_n)$ such that
\begin{enumerate}
\item[(a)] 
$0 \le p_1 < m/g$, 
\item[(b)] 
for each $1 < i < n$, $0 \le p_i < m$, and
\item[(c)] 
$0 \le p_n < vm/g$.
\end{enumerate}
\item[(2)]
$w = m^n/g^2$.
\item[(3)] 
$w\neq 0$ in $\Bbbk$.
\end{enumerate}
\end{lemma}

\begin{proof} The following proof is similar to the
proof of Lemma \ref{xxlem4.1}.

(1) Pick an arbitrary coset $M'$ of $M$, and let ${\bf p} = 
(p_1, ..., p_n) \in M'$. Since 
$$(m/g, -m/g, m/g, -m/g, ..., m/g, -m/g) \in M,$$
we can replace $p_1$ by $r$ where 
$0\le r\le g$ and $r \equiv p_1$ mod $m/g$ within the coset
$M'$. Therefore we can assume, without loss of generality, 
that $0 \le p_1 < m/g$. Furthermore, if ${\bf t} \in M$ then 
$t_1 \in (m/g) \Z$, so there is no vector in $M'$ whose 
first component is any other $0 \le r' < m/g$. 

For each $1 < i < n$, $m ({\bf e}_i - {\bf e}_n) \in M$, 
so we can assume, without loss of generality, that 
$0 \le p_i < m$. Furthermore, if ${\bf t} \in M$ and 
$t_1 = 0$, then each other $t_i \in m\Z$. This implies
that there is no other set of possible values of 
$p_1, ..., p_{n-1}$ subject to the conditions 
$0 \le p_1 < m/g$ and $0 \le p_i < m$ for every $1<i<n$. 

Finally, $(vm/g) {\bf e}_n \in M$, so there exists a 
representative ${\bf p} \in M'$ subject to 
constraints (a)-(c) of the lemma. If $c {\bf e}_n \in M$, 
then $c \in m \Z \cap v \Z = (vm/g) \Z$, so this 
representative is unique. This finishes the proof of
part (1).

(2) The value $w$ can be determined by counting the cosets 
by their representatives. For every sequence of integers 
$p_1, ..., p_{n-1}$ such that $0 \le p_1 < m/g$ and 
$0 \le p_i < m$ for all $1<i<n$, there are $m/g$ 
possible values of $p_n$ such that $0 \le p_n < vm/g$ 
and $(p_1, ..., p_n) \in H_v$. Therefore, 
$w = (m/g) \cdot m^{n-2} \cdot (m/g) = m^{n}/g^2$. 

(3) Since $q\in \Bbbk$ and $o(q)=m$, then the characteristic
of $\Bbbk$ cannot divide $m$. Consequently $m\neq 0$ and $w\neq 0$
in $\Bbbk$.
\end{proof}

\begin{theorem}
\label{xxthm5.3}
Let $B=\Bbbk_q [\bfx]^{(v)}$ when $n$ is even and
let $R$ be the center of $B$. Suppose that $m$ is invertible 
in $\Bbbk$ and that $v$ divides 
$wp(\frac{m}{g}-1)$. Then
$$d_w^{[p]}(B/R) =_{\Bbbk^{\times}} 
(x_1 x_2 \cdots x_n)^{wp(\frac{m}{g}-1)}
=_{\Bbbk^{\times}} 
(x_1^v x_2^v \cdots x_n^v)^{\frac{wp}{v}(\frac{m}{g}-1)}.
$$
As a consequence, $d_w^{[p]}(B/R)$ is stable.
\end{theorem}

\begin{proof} This proof is similar to the proof of
Theorem \ref{xxthm4.4}.

By Lemmas \ref{xxlem2.4}(4)
and \ref{xxlem4.3}, $d_w^{[p]}(B/R) =\gcd \Lambda^{2p}$
where 
$$\Lambda:=\left\{ \prod_{i=1}^w \bfx^{ {\bf s}_i} 
\, \middle| \, {\bf s}_i\in \N^n \cap ( M + {\bf b}_i )\right\}.
$$
For each $1 \le s \le n$, let $f_s \in \N$ be maximal such 
that $x_s^{f_s}$ divides all elements of $\Lambda$. 
This gives $\bfx^{2p{\bf f}}$ as the gcd of $\Lambda^{2p}$ 
in the over-algebra $\Bbbk_q[\bfx]\supseteq B$ where ${\bf f}=
(f_1,\cdots,f_n)$. If $2p {\bf f} \in H_v^{+}$, then it is the gcd 
of $\Lambda^{2p}$ in $\Bbbk_q[\bfx]^{(v)}$ as well.

By symmetry (see the proof of Theorem \ref{xxthm4.4} for a similar argument),  
$f_1=f_2=\cdots =f_n$, and we will only
work out $f_1$. We calculate $f_1$ by summing the lowest 
powers of $x_1$ in each coset of $M$ (or more precisely, 
in each $\N^n \cap ( M + {\bf b}_i )$ for different 
$i$). These lowest powers can be found by using the 
representatives outlined in Lemma \ref{xxlem5.2}, 
which also shows that this power cannot exceed 
$m/g - 1$. For each $0 \le k \le m/g-1$, there are 
$m^{n-1}/g$ cosets with lowest power $x_1^k$. Therefore, 
the sum is 
$$f_1 = \frac{m^{n-1}}{g} \; \frac{(m/g)(m/g-1)}{2} = \frac{w}{2}\left(\frac{m}{g}-1\right).$$ 
Now $2p{\bf f}=wp(m/g-1) (1,1,\cdots,1)\in H_v^{+}$ as $v$ 
divides $wp(m/g-1)$. 
The assertion follows from the last paragraph, and stability of
$d_w^{[p]}(B/R)$ follows clearly from the main assertion.
\end{proof}

\section{Application I: automorphism group}
\label{xxsec6}

For any algebra $A$, let $\Aut(A)$ denote the group
of all algebra automorphisms of $A$. When $A$ is 
${\mathbb N}$-graded, let $\Aut_{gr}(A)$ denote the group
of all graded algebra automorphisms of $A$. 

In this section we only consider the algebra 
$A:=k_q[{\bf x}]^{(v)}$ and use 
$g$ for an algebra automorphism of $A$. 
First we consider an algebra automorphism $g$ satisfying
\begin{equation}
\label{E6.0.1}\tag{E6.0.1}
g((x_1^v \cdots x_n^v)^a)=_{\Bbbk^{\times}} 
(x_1^v \cdots x_n^v)^a, \quad {\text{for some positive integer $a$.}}
\end{equation}
The first few lemmas discuss some easy  
properties of $g$ satisfying \eqref{E6.0.1}.

There is a natural $\N^n$-grading on the skew polynomial
ring $\Bbbk_q[{\bf x}]$ with $\deg x_i={\bf e}_i$ for
$i=1,\cdots,n$. We consider $\Bbbk_q[{\bf x}]^{(v)}$ as
an $\N^n$-graded subalgebra of $\Bbbk_q[{\bf x}]$. 
Both $\Bbbk_q[{\bf x}]$ and $\Bbbk_q[{\bf x}]^{(v)}$
are also $\N$-graded by considering the total degree. We
will use both gradings in this section.

For any permutation $\pi$ of $\{1, ..., n\}$, 
we denote the linear function 
$\boldsymbol\pi: \mathbb{Z}^n \rightarrow 
\mathbb{Z}^n$ 
determined by 
$\pi: {\bf e}_i \mapsto {\bf e}_{\pi(i)}$.
For a permutation $\pi\in S_n$, we have
\begin{equation}
\label{E6.0.2}\tag{E6.0.2}
\bfx^{\pi(\bfs)}=x_1^{s_{\pi^{-1}(1)}}\cdots x_{n}^{s_{\pi^{-1}(n)}}
\end{equation}
and denote
\begin{equation}
\label{E6.0.3}\tag{E6.0.3}
\bfx^{\bfs}_{\pi}: =x_{\pi(1)}^{s_1}\cdots x_{\pi(n)}^{s_n}.
\end{equation}
It is clear that $\bfx^{\pi(\bfs)}=_{\Bbbk^{\times}} \bfx^{\bfs}_{\pi}$.

\begin{lemma}
\label{xxlem6.1} Let $g\in \Aut(\Bbbk_q[{\bf x}]^{(v)})$
satisfying \eqref{E6.0.1} in parts {\rm{(1)-(4)}}.
\begin{enumerate}
\item[(1)]
The image of every monomial through $g$ is 
a $\Bbbk^\times$-multiple of a monomial.
\item[(2)]
$\deg g(f)=\deg f$ for any monomial $f$. As a consequence,
$g$ is a graded algebra automorphism.
\item[(3)]
The image of each $x_i^v$ is a $\Bbbk^\times$-multiple 
of some $x_j^v$.
\item[(4)]
There exists a permutation $\pi_g$ of $\{1, ..., n\}$ 
such that each monomial ${\bf x}^{\bf s}$ is mapped 
to a $\Bbbk^\times$-multiple of ${\bf x}^{\pi_g({\bf s})}$.
\item[(5)]
$g$ satisfies \eqref{E6.0.1} if and only if, for each $i$,
there is a $j$ such that 
$$g(x_i^v)=_{\Bbbk^{\times}} x_j^v.$$
\end{enumerate}
\end{lemma}

\begin{proof}
(1) By \eqref{E6.0.1}, $g((x_1 x_2 \cdots x_n)^{vaN})=_{\Bbbk^{\times}}
(x_1 x_2 \cdots x_n)^{vaN}$ for all $N>0$. 
Let $f$ be any monomial in $\Bbbk_q[{\bf x}]^{(v)}$. Then
$f$ is a factor of $(x_1 x_2 \cdots x_n)^{vaN}$ for some 
$N>0$. Let $f'$ be a monomial such that 
$ff'=_{\Bbbk^{\times}}(x_1 x_2 \cdots x_n)^{vaN}$. Then 
$$g(f) g(f')=g((x_1 x_2 \cdots x_n)^{vaN})=_{\Bbbk^{\times}}
(x_1 x_2 \cdots x_n)^{vaN}.$$
Since $\N^n$ is an ordered semigroup and 
$\Bbbk_q[{\bf x}]^{(v)}$ is an $\N^n$-graded domain,
both $g(f)$ and $g(f')$ are $\N^n$-homogeneous. Every
$\N^n$-homogeneous element is a $\Bbbk^\times$-multiple 
of a monomial. The assertion follows.

(2) Note that the lowest total degree of a non-scalar 
element in $\Bbbk_q[{\bf x}]^{(v)}$ is $v$. Applying 
$g$ to the monomials $f$ of (total) degree $v$, we have that
$\deg g(f)\geq v=\deg f$. Since every monomial in 
$\Bbbk_q[{\bf x}]^{(v)}$ is a product of monomials of 
degree $v$, $\deg g(f)\geq \deg f$ for all monomials. 
By symmetry, $\deg g^{-1}(f)\geq \deg f$. The assertion
follows.

(3) If $f$ is a degree $v$ monomial of $\Bbbk_q[{\bf x}]^{(v)}$, 
$f^2$ can be decomposed as 
$$f^2 =_{\Bbbk^\times} f_1 f_2$$ 
where $f_1, f_2$ are degree $v$ monomials. 
The decomposition is unique if and only if 
$f= x_i^v$ for some $i$. This property is invariant under $g$.

(4) We choose $\pi$ so that, for each $i$, we have 
$g(x_i^v) =_{\Bbbk^\times} x_{\pi(i)}^v$ by part (3). 
For any monomial ${\bf x}^{\bf s}$, we have 
\begin{equation}
\label{E6.1.1}\tag{E6.1.1}
g({\bf x}^{\bf s})^v = g({\bf x}^{v\bf s}) 
=_{\Bbbk^\times} x_{\pi(1)}^{vs_1} \cdots x_{\pi(n)}^{vs_n} 
=_{\Bbbk^\times} ({\bf x}^{\pi_g({\bf s})})^v,
\end{equation}
which implies that 
$g({\bf x}^{\bf s}) 
=_{\Bbbk^\times} {\bf x}^{\pi_g({\bf s})}.$

(5) One implication is part (3) and the other implication
is clear.
\end{proof}

Next we wish to understand the coefficients of the image
of $g$. The next lemma deals with the case when $\pi_g$
is the identity. For any automorphism $g$ of 
$\Bbbk_q[{\bf x}]^{(v)}$, we say $\pi_g=1$ if 
$$g(x_i^v)=_{\Bbbk^{\times}} x_i^v$$
for all $i=1,\cdots,n$. Let $\Aut_{1}(\Bbbk_q[{\bf x}]^{(v)})$
be the subgroup of $\Aut(\Bbbk_q[{\bf x}]^{(v)})$ 
consisting of automorphisms $g$ with $\pi_g=1$. 
It is clear that $\Aut_{1}(\Bbbk_q[{\bf x}]^{(v)})\subseteq
\Aut_{gr}(\Bbbk_q[{\bf x}]^{(v)})$.

\begin{lemma}
\label{xxlem6.2} Retain the above notation.
\begin{enumerate}
\item[(1)]
Let $g \in \Aut_{1}(\Bbbk_q[{\bf x}]^{(v)})$.
Then there exist $(c, k_2, ..., k_n)\in (\Bbbk^\times)^{n}$ such that 
for each monomial ${\bf x}^{\bf s}$ of degree $Nv$, 
\begin{equation}
\label{E6.2.1}\tag{E6.2.1}
g({\bf x}^{\bf s}) = c^N k_2^{s_2} \cdots k_n^{s_n} {\bf x}^{\bf s}.
\end{equation}
\item[(2)]
Conversely, given $(c, k_2, ..., k_n)\in (\Bbbk^\times)^{n}$, then \eqref{E6.2.1}
defines a unique algebra automorphism $g \in\Aut_{1}(\Bbbk_q[{\bf x}]^{(v)})$.
\end{enumerate}
As a consequence, $\Aut_{1}(\Bbbk_q[{\bf x}]^{(v)})\cong (\Bbbk^{\times})^{n}$.
\end{lemma}
\begin{proof}
(1) Let $c$ be such that $g(x_1^v) = c x_1^v$, and let $k_1 = 1$. 
For each $i \ne 1$, let $k_i$ be such that 
$g(x_1^{v-1} x_i) = c k_i x_1^{v-1} x_i$. For any monomial 
${\bf x}^{\bf s}$ of degree $Nv$ (which means that it is in
$\Bbbk_q[{\bf x}]^{(v)}$), there exists a scalar $r \in \Bbbk^{\times}$ 
such that
$$x_1^{Nv(v - 1)} {\bf x}^{\bf s} = 
r(x_1^{v-1} x_1)^{s_1} \cdots (x_1^{v-1} x_n)^{s_n}$$
Therefore
$$\begin{aligned}
c^{N(v-1)} x_1^{Nv(v - 1)} g({\bf x}^{\bf s}) 
&= g(x_1^{Nv(v - 1)} {\bf x}^{\bf s})\\
&= r(c k_1 x_1^{v-1} x_1)^{s_1} \cdots (c k_n x_1^{v-1} x_n)^{s_n} \\
&= c^{Nv} k_1^{s_1} \cdots k_n^{s_n}x_1^{Nv(v - 1)} {\bf x}^{\bf s},
\end{aligned}
$$
which implies that
$$g({\bf x}^{\bf s}) = c^N k_2^{s_2} \cdots k_n^{s_n} {\bf x}^{\bf s}.$$

(2) This is easy and the proof is omitted.
\end{proof}

For any $g_1, g_2 \in \Aut(\Bbbk_q[{\bf x}]^{(v)})$ 
such that $\pi_{g_1} = \pi_{g_2}$, we have 
$\pi_{g_1^{-1} \circ g_2} = 1$, and there exist 
$c, k_1=1, k_2, ..., k_n \in \Bbbk^\times$ such that 
for any monomial ${\bf x}^{\bf s}$ of degree $Nv$,
\begin{equation}
\label{E6.2.2}\tag{E6.2.2}
g_2({\bf x}^{\bf s}) = c^N k_1^{s_1} \cdots k_n^{s_n} 
g_1({\bf x}^{\bf s}).
\end{equation}

The automorphism group can therefore be fully 
determined by determining the possible values 
of $\pi_g$ and producing an example automorphism for each.
We discuss possible $\pi_g$ in the next lemma.

\begin{lemma}
\label{xxlem6.3} Let $g$ denote an automorphism 
of $\Bbbk_q[{\bf x}]^{(v)}$ satisfying \eqref{E6.0.1}.
\begin{enumerate}
\item[(1)]
If $q = \pm 1$, for every permutation $\pi$ of 
$\{1, ..., n\}$, there exists $g$ such that 
$\pi_g = \pi$, and for each ${\bf x}^{\bf s}$, 
$g({\bf x}^{\bf s}) = {\bf x}^{\bf s}_{\pi_g}$.
\item[(2)]
If $q^v \ne \pm 1$, then, for any $g$, we have
$\pi_g = 1$. 
\item[(3)]
If $q^v = \pm 1$, then, for each $m \in \mathbb{Z}$, 
there exists $g$, such that $\pi_g$ is addition 
by $m$ modulo $n$, 
and  $g({\bf x}^{\bf s}) = {\bf x}^{\bf s}_{\pi_g}$.
\item[(4)]
If $q^v = \pm 1$ and $q \ne \pm 1$, then, for any $g$, 
there exists $m \in \mathbb{Z}$ such that $\pi_g$ 
is addition by $m$ modulo $n$. 
\end{enumerate}
\end{lemma}

\begin{proof}
(1) The relations of $\Bbbk_q[{\bf x}]$ are simply 
$x_i x_j = q x_j x_i$ for all $i \ne j$. Therefore
any permutation $\pi$ of the generators $x_1,...,x_n$ extends
to an automorphism $g$ of $\Bbbk_q[{\bf x}]$, and $g$
restricts to an automorphism of $\Bbbk_q[{\bf x}]^{(v)}$.

(2) For any distinct $i, j$, we have 
\begin{equation}
\label{E6.3.1}\tag{E6.3.1}
(x_i^{v-1} x_j) x_i^v = r_{i,j} x_i^v (x_i^{v-1} x_j)
\end{equation}
where
$$r_{i,j} = \left\{
\begin{array}{ll}
q^{-v} & j < i,\\
q^v & i < j.\\
\end{array}
\right.$$
We apply $g$ to both sides of \eqref{E6.3.1}. Then Lemma 
\ref{xxlem6.1}(5) shows that $r_{i,j} = r_{\pi(i),\pi(j)}$, where $\pi=\pi_g$. 
Since $q^v \ne q^{-v}$, we have that $i < j$ implies $\pi(i) < \pi(j)$.
Therefore $\pi$ is the identity.

(3) It suffices to prove the assertion in the case 
$m = 1$. Let $\bfs=(s_1,\cdots,s_n)$ and $\pi(i) \equiv i + 1$ mod $n$. Then
$$\bfx^{\bfs}_{\pi}=x_{2}^{s_1}\cdots x_{n}^{s_{n-1}} x_1^{s_n}
=q^{\sum_{i=1}^{n-1} s_n s_i} x_1^{s_n}x_{2}^{s_1}\cdots x_{n}^{s_{n-1}}
=q^{\sum_{i=1}^{n-1} s_n s_i} \bfx^{\pi(\bfs)}.$$
For all $\bfs$ and ${\bf t}$,
$$\bfx^{\bfs} \bfx^{\bf t}=q^{\sum_{i<j} s_j t_i} \bfx^{\bfs +{\bf t}}.$$

Let $\alpha({\bf s})=s_n^2$ and define
$$g: {\bf x}^{\bf s} \mapsto q^{\alpha({\bf s})} {\bf x}^{\bf s}_{\pi}$$
for all monomials ${\bf x}^{\bf s}$ in $\Bbbk_q[{\bf x}]$.
Note that $g$ cannot extend to an automorphism 
of $\Bbbk_q[{\bf x}]$. But we show next that $g$ extends
to an automorphism of $\Bbbk_q[{\bf x}]^{(v)}$.

To show this, it suffices to show 
that 
$$g(\bfx^{\bfs}) g(\bfx^{\bf t})=
g(\bfx^{\bfs}\bfx^{\bf t})
(=q^{\sum_{i<j} s_j t_i} g(\bfx^{\bfs +{\bf t}}))$$
for all $\bfx^{\bfs}$ and $\bfx^{\bft}$ in $\Bbbk_q[{\bf x}]^{(v)}$. 
Using the above computation, we have
$$\begin{aligned}
g(\bfx^{\bfs} \bfx^{\bf t})&= q^{\sum_{i<j} s_j t_i} g(\bfx^{\bfs +{\bf t}})\\
&=q^{\alpha(\bfs+\bft)}q^{\sum_{i<j} s_j t_i} q^{\sum_{i=1}^{n-1} (s_n+t_n)(s_i+t_i)}
\bfx^{\pi (\bfs+\bft)},
\end{aligned}
$$
and
$$\begin{aligned}
g(\bfx^{\bfs}) g(\bfx^{\bf t})&=
q^{\alpha(\bfs)+\alpha(\bft)} q^{\sum_{i=1}^{n-1} s_n s_i}
q^{\sum_{i=1}^{n-1} t_n t_i} \bfx^{\pi(\bfs)} \bfx^{\pi(\bft)}\\
&=q^{\alpha(\bfs)+\alpha(\bft)} q^{\sum_{i=1}^{n-1} s_n s_i}
q^{\sum_{i=1}^{n-1} t_n t_i} q^{\sum_{i<j} \pi(\bfs)_j \pi(\bft)_i}
\bfx^{\pi(\bfs+\bft)} .
\end{aligned}
$$

By direct calculation, the difference between the $q$-powers 
in the expressions of $g(\bfx^{\bfs} \bfx^{\bf t})$ and 
$g(\bfx^{\bfs}) g(\bfx^{\bf t})$ is $2 s_n (\sum_{i=1}^n t_i)$.
Since $v$ divides $\sum_{i=1}^n t_i$ and
$q^v=\pm 1$, we have $q^{2 s_n (\sum_{i=1}^n t_i)}=(\pm 1)^2=1$.
Therefore $g(\bfx^{\bfs} \bfx^{\bf t})=g(\bfx^{\bfs}) g(\bfx^{\bf t})$
so $g$ is an algebra automorphism.

(4) For distinct $i, j$, let $y_{i,j} = x_i^{v-1} x_j$. For any distinct 
$i, j, k$, we have
$$y_{i,j} y_{i,k} = r y_{i,k} y_{i,j},$$ 
where
$$r = 
\left\{
\begin{array}{ll}
q^{-1} & \quad {\text{if  }} i < j < k \text{ or } j < k < i \text{ or } k < i < j ,\\
q & \quad {\text{if  }} i < k < j \text{ or } k < j < i \text{ or } j < i < k .\\
\end{array}
\right.
$$
Recall $q \ne q^{-1}$. For any $i, j$, the number of values of $k$ that 
yield $r = q$ is equal to $j - i - 1$ mod $n$. 
Since this is true for all $i\ne j$, we have $\pi_g(j) - \pi_g(i)-1 \equiv
j-i-1$ mod $n$. Therefore $\pi_g(j) - \pi_g(i) \equiv j - i$ 
mod $n$, and the assertion follows by letting $m = \pi_g(n)$. 
\end{proof}

We are now ready to prove Theorems \ref{xxthm0.1} and \ref{xxthm0.2}.

\begin{proof}[Proof of Theorem \ref{xxthm0.1}]
For each $\sigma\in S_n$, let $F_{\sigma}$ be the 
algebra automorphism of $\Bbbk_{-1}[\bfx]$ induced by
sending $x_i$ to $x_{\sigma(i)}$ for all $i$. This 
automorphism restricts to an algebra automorphism 
of $\Bbbk_{-1}[\bfx]^{(v)}$, which is still denoted
by $F_{\sigma}$ -- see Lemma \ref{xxlem6.3}(1). Then 
the subgroup generated by all
$\{F_{\sigma} \mid \sigma\in S_n\}$ is isomorphic to
$S_n$.

Now assume that $n$ and $v$ have different parity
and that $g$ is an algebra automorphism of 
$\Bbbk_{-1}[\bfx]^{(v)}$. Recall that $m=2$.
If $n$ is odd, $\gcd(m,v)=2$
and we can apply Theorem \ref{xxthm4.4}. If $n$ is even,
$\gcd(m,v)=1$, so we can apply Theorem \ref{xxthm5.3}.
In both cases, by Theorem \ref{xxthm4.4} or \ref{xxthm5.3},
the $v$-power discriminant $d^{[v]}_{w}(\Bbbk_{-1}[\bfx]^{(v)}/R)$
is of the form $(x_1^v \cdots x_n^v)^{N}$ for some $N>0$. 
By Lemma \ref{xxlem2.1}(1), this
discriminant is $g$-invariant. This means that 
$g$ satisfies \eqref{E6.0.1}. Let $\pi_g$ be the permutation 
defined in Lemma \ref{xxlem6.1}(4). It is easy to 
see that the map $\phi:g\to F_{\pi_g}$ is a surjective group 
homomorphism from $\Aut(\Bbbk_{-1}[\bfx]^{(v)})$ to $S_n$
with kernel being $\Aut_1(\Bbbk_{-1}[\bfx]^{(v)})$.
By Lemma \ref{xxlem6.2}, we have $\Aut_1(\Bbbk_{-1}[\bfx]^{(v)})
\cong (\Bbbk^{\times})^n$. Therefore
$$\Aut(\Bbbk_{-1}[\bfx]^{(v)})\cong S_n \ltimes
\Aut_1(\Bbbk_{-1}[\bfx]^{(v)})\cong 
S_n \ltimes (\Bbbk^{\times})^n.$$
\end{proof}

The proof of Theorem \ref{xxthm0.2} is similar.

\begin{proof}[Proof of Theorem \ref{xxthm0.2}]
The proofs of (1) and (2) are similar, we only 
provide the proof of (2) here.
 
(2) Under hypotheses (a) or (b), we use 
Theorem \ref{xxthm4.4} or \ref{xxthm5.3} to conclude 
that 
the $v$-power discriminant $d^{[v]}_{w}(\Bbbk_{q}[\bfx]^{(v)}/R)$
is of the form $(x_1^v \cdots x_n^v)^{N}$ for some $N>0$. 
By Lemma \ref{xxlem2.1}(1), this
discriminant is $g$-invariant. This means that 
$g$ satisfies \eqref{E6.0.1}. By Lemma \ref{xxlem6.3}(2), we have
$\pi_g=1$, or equivalently, $g\in 
\Aut_1(\Bbbk_{q}[\bfx]^{(v)})$.
Therefore
$$\Aut(\Bbbk_{q}[\bfx]^{(v)})=
\Aut_1(\Bbbk_{q}[\bfx]^{(v)})\cong 
(\Bbbk^{\times})^n.$$
\end{proof}

\section{Application II: cancellation problem}
\label{xxsec7} 

The second application of the discriminant method is 
the cancellation problem. We need to recall some
definitions and results from \cite{BZ}.

\begin{definition} \cite[Definition 1.1]{BZ}
\label{xxdef7.1}
Let $A$ be an algebra.
\begin{enumerate}
\item[(1)]
We call $A$ {\it cancellative} if 
$A[t] \cong B[t]$ for some algebra $B$ implies that 
$A \cong B$.
\item[(2)]
We call $A$ {\it strongly cancellative} if, for any
$d\geq 1$, $A[t_1,\cdots,t_d]\cong B[t_1,\cdots,t_d]$
for some algebra $B$ implies that $A\cong B$.
\item[(3)]
We call $A$ {\it universally cancellative} if, for any 
$\Bbbk$-flat finitely generated commutative domain 
$R$ such that $R/I=\Bbbk$ for some ideal $I \subset R$ 
and any $\Bbbk$-algebra $B$, any algebra isomorphism 
$A\otimes R\cong B\otimes R$ implies that $A\cong B$.
\end{enumerate}
\end{definition}

The first result is 

\begin{lemma}\cite[Proposition 1.3]{BZ}
\label{xxlem7.2}
Let $\Bbbk$ be a field and $A$ be an algebra with center 
$C(A) = \Bbbk$. Then $A$ is universally cancellative, hence,
strongly cancellative.
\end{lemma}

We only need the following definition 
for connected graded domains.

\begin{definition}\cite[Definition 2.1(2)]{CPWZ1}
\label{xxdef7.3}
Let $A$ be a connected graded domain generated 
by $A_1=\bigoplus_{i=1}^r \Bbbk x_i$. An element $f \in A$ 
is called dominating if, for every testing 
${\mathbb N}$-filtered PI algebra $T$ with $\gr_F T$ being a
connected graded domain, and for every testing subset 
$\{y_1,\cdots, y_r\}\subseteq T$
that is linearly independent in the quotient $\Bbbk$-module 
$T/F_0T$, there is a
presentation of $f$ of the form $f(x_1,\cdots,x_r)$ in the 
free algebra $\Bbbk\langle x_1,\cdots, x_r\rangle$
such that the following hold: either $f(y_1, \cdots, y_r)=0$, or
\begin{enumerate}
\item[(a)] 
$\deg f(y_1,\cdots,y_r)\geq \deg f$, and
\item[(b)] 
$\deg f(y_1,\cdots, y_r)> \deg f$, further, 
$\deg y_{i_0} > 1$ for some $i_0$.
\end{enumerate}
\end{definition}

\begin{lemma}\cite[Theorem 4.6]{BZ}
\label{xxlem7.4}
Let $A$ be a connected graded PI domain generated in degree 1,
of finite Gelfand-Kirillov dimension. Suppose that the discriminant 
power $(d^{[p]}_w(A/C))^a$ is dominating for some $p,w$ and $a$. 
Then $A$ is strongly cancellative.
\end{lemma}

\begin{proof} The original \cite[Theorem 4.6]{BZ} was proved
for discriminant $d_w(A/C)$. But the proof works for this
more general setting when \cite[Lemma 4.5(2)]{BZ}
is replaced by Lemma \ref{xxlem2.2}. So we are not going to 
repeat the rest of the proof.
\end{proof}

The following lemma is easy.

\begin{lemma}
\label{xxlem7.5} 
Let $A$ be the algebra $\Bbbk_q[\bfx]^{(v)}$ for
some $n,q,v$. Let $f$ be an element of the form
$(x_1\cdots x_n)^N$ for some $N>0$. Then there is an integer 
$a>0$ such that $f^{ab}$ is dominating for 
all integer $b>0$.
\end{lemma}

\begin{proof}
Let $\Phi:=\{x_1^{d_1}\cdots x_n^{d_n}\mid d_s\geq 0, \sum_{s=1}^n d_s=v\}$
be the set of monomials of degree $v$, which is a $\Bbbk$-basis 
of the degree 1 component of $A$ after regrading. Let $P$ be the 
product of elements in $\Phi$. Then $P=_{\Bbbk^{\times}} (x_1\cdots x_n)^a$
for some $a>0$. Then $f^{ab}=_{\Bbbk^{\times}} P^{bN}$. It 
suffices to show that $P^{bN}$ is dominating. But
this is \cite[Lemma 2.2(1)]{CPWZ1}.
\end{proof}

Now we are ready to prove Theorem \ref{xxthm0.4}. In fact
we prove that the algebras are strongly cancellative.

\begin{theorem}
\label{xxthm7.6} 
Let $A$ be $\Bbbk_{q}[x_1,\cdots,x_n]^{(v)}$ where $v$ is a 
positive integer and let $m\geq 2$ be the order of $q$.
Suppose that one of the following is true.
\begin{enumerate}
\item[(a)] 
$n$ is even and $m$ does not divide $v$.
\item[(b)]
$n$ is odd and $\gcd(m,v)\neq 1$. 
\end{enumerate}
Then $A$ is strongly cancellative.
\end{theorem}

\begin{proof} Under the hypotheses (a) or (b), by 
Theorems \ref{xxthm4.4} and \ref{xxthm5.3},
there is some $p$ and $w$ such that 
$d^{[p]}_w(A/C)$ is of the form $(x_1\cdots x_n)^N$
for some $N>0$. By Lemma \ref{xxlem7.5}, the element
$(d^{[p]}_w(A/C))^{ab}$ is dominating for 
some $a>0$ and all $b>0$. The assertion follows from 
Lemma \ref{xxlem7.4}.
\end{proof}

We make some comments and remarks for the rest of this section.

\begin{lemma}
\label{xxlem7.7}
Let $\{A^1,\cdots,A^s\}$ be a set of algebras as in Theorem 
{\rm{\ref{xxthm7.6}(a,b)}} with possible repetition. Let $A$ be the 
tensor product $A^1\otimes \cdots \otimes A^s$. Then 
some $p$-power discriminant of $A$ over its center is dominating.
\end{lemma}

\begin{proof} 
Each algebra $A^i$ has some $p$-power discriminant 
(over its center) that is dominating by Theorems 
\ref{xxthm4.4} and \ref{xxthm5.3}. The assertion follows
from Lemma \ref{xxlem2.6}(3)  together with induction.
Some of the hypotheses in Lemma \ref{xxlem2.6} can be
verified by using Lemma \ref{xxlem3.3}.
\end{proof}

\begin{remark}
\label{xxrem7.8} 
Let $A$ be as in Lemma {\rm{\ref{xxlem7.7}}}.
\begin{enumerate}
\item[(1)]
By using the discriminant method \cite{CPWZ1, CPWZ2}, we 
obtain that every automorphism of $A$
is graded. Therefore it is a linear algebra problem to determine
the full automorphism group of $A$. In many case (when 
the Gelfand-Kirillov dimension of $A$ is small), one can 
explicitly work out the full automorphism group of $A$.
\item[(2)]
By Lemma \ref{xxlem7.4} and \ref{xxlem7.7}, $A$ is strongly 
cancellative.
\end{enumerate}
\end{remark}

\section{Tits alternative}
\label{xxsec8}

Recall that, in the last few sections, we are only 
considering the case when $q\neq 1$, which implies that
\begin{equation}
\label{E8.0.1}\tag{E8.0.1}
\Bbbk\neq {\mathbb Z}/(2).
\end{equation}
In this section, if $q=1$, we will further assume
that $\Bbbk\neq {\mathbb Z}/(2)$. Note that \eqref{E8.0.1}
is one of the hypotheses in \cite[Proposition 2.5]{CPWZ3}.

Firstly we consider the case when $n=2s+1$ is odd and 
$g:=\gcd(m,v)=1$ where $m\geq 2$ is the order of $q$.
Since $\gcd(m,v)=1$, there are two positive integers 
$\alpha$ and $\beta$ such that
\begin{equation}
\label{E8.0.2}\tag{E8.0.2}
(\alpha +s)m-\beta v=1.
\end{equation}

\begin{lemma}
\label{xxlem8.1} Retain the above hypotheses. 
\begin{enumerate}
\item[(1)]
The following are locally nilpotent derivations of $\Bbbk_{q}[\bfx]$
of degree $\beta v$.
\begin{enumerate}
\item[(a)]
$$\partial_1: x_i\longrightarrow 
\begin{cases} x_2^{\alpha m} (x_2x_3^{m-1}x_4 x_5^{m-1}\cdots x_{2s}
x_{2s+1}^{m-1}) & i=1,\\
0& i\neq 1.\end{cases}$$
\item[(b)]
$$\partial_3: x_i\longrightarrow 
\begin{cases} x_2^{\alpha m} (x_1^{m-1}x_2x_4 x_5^{m-1}\cdots x_{2s}
x_{2s+1}^{m-1}) & i=3,\\
0& i\neq 3.\end{cases}$$
\end{enumerate}
\item[(2)]
Let $g_1=\exp(\partial_1)$ and $g_3=\exp(\partial_3)$.
Then $g_1$ and $g_3$ are two automorphisms of $\Bbbk_{q}[\bfx]$
that generate a free subgroup of $\Aut(\Bbbk_{q}[\bfx])$.
\item[(3)]
Both $g_1$ and $g_2$ send a homogeneous element $f$ of
degree $h$ to a linear combination of homogeneous
elements of degrees $h+\beta v {\mathbb N}$. As a consequence,
both $g_1$ and $g_2$ restrict to algebra automorphisms of 
$\Bbbk_{q}[\bfx]^{(v)}$.
\end{enumerate}
\end{lemma}

\begin{proof} (1) (a)
The degree of $\partial_1$ is $\alpha m+ sm-1=\beta v$. To check
that $\partial_1$ is a derivation we just verify that
$\partial_1(x_j x_i-q x_i x_j)=0$ for all $1\leq i<j\leq n$,
which is straightforward by the choice of $\partial_1(x_i)$.
It is clear that $\partial^2_1(x_i)=0$. Then $\partial^2_1$
is locally nilpotent. The proof of (1)(b) is similar.

(2) By definition, for any power $d$, we have 
$$\begin{aligned}
g_1^d(x_i)&=\begin{cases} x_1+d\; x_2^{\alpha m} (x_2x_3^{m-1}x_4 x_5^{m-1}\cdots x_{2s}
x_{2s+1}^{m-1}) & i=1,\\
x_i& i\neq 1,\end{cases}\\
&=\begin{cases} x_1+d\; x_3 [x_2^{\alpha m} (x_2x_3^{m-2}x_4 x_5^{m-1}\cdots x_{2s}
x_{2s+1}^{m-1})] & i=1,\\
x_i& i\neq 1,\end{cases}
\end{aligned}
$$
and
$$\begin{aligned}
g_3^d(x_i)&=\begin{cases} x_3+d\; x_2^{\alpha m} (x_1^{m-1}x_2x_4 x_5^{m-1}\cdots x_{2s}
x_{2s+1}^{m-1}) & i=1,\\
x_i& i\neq 1,\end{cases}\\
&=\begin{cases} x_1+d \; x_1 [x_2^{\alpha m} (x_1^{m-2}x_2x_4 x_5^{m-1}\cdots x_{2s}
x_{2s+1}^{m-1})] & i=1,\\
x_i& i\neq 1.\end{cases}
\end{aligned}
$$
Let $R$ be the subalgebra of $\Bbbk_{q}[\bfx]$ generated by $x_2,x_4,x_5,\cdots,x_{n}$.
Then $g_1$ satisfies \cite[(E2.1.1)]{CPWZ3} with $a_0=0$ and
$a_1=x_2^{\alpha m} (x_2x_3^{m-2}x_4 x_5^{m-1}\cdots x_{2s}
x_{2s+1}^{m-1})$ and $g_3$ satisfies
\cite[(E2.1.2)]{CPWZ3} (when $x_2$ is replaced by $x_3$)
with $b_0=0$ and $b_1=x_2^{\alpha m} (x_1^{m-2}x_2x_4 x_5^{m-1}\cdots x_{2s}
x_{2s+1}^{m-1})$. It is clear that $a_1b_1$ is 
transcendental over $\Bbbk$. Also $R+Rx_1+Rx_3$ is a free $R$-module of rank 3. 
Thus we have checked all hypotheses of \cite[Proposition 2.5]{CPWZ3}.
By \cite[Proposition 2.5(2)]{CPWZ3} and its proof, the subgroup generated by
$g_1$ and $g_3$ is free.

(3) Since $\partial_1$ has degree $\beta v$, the first assertion follows
because $g_1=\exp(\partial_1)$. In particular, $g_1$ maps a homogeneous 
element of degree $v$ to a linear combination of homogeneous
elements of degrees in $v+\beta v {\mathbb N}$. Thus $g_1$ restricts 
to an automorphism of $\Bbbk_{q}[\bfx]^{(v)}$. The same statement holds for
$g_3$.
\end{proof}

\begin{lemma}
\label{xxlem8.2}
Let $A$ be a connected graded domain generated 
in degree one. Let $g$ be an automorphism of $\Bbbk_{q}[\bfx]$ such that
\begin{enumerate}
\item[(a)]
$g(x)=x+{\sf{higher\;\; degree\;\; terms}}$ for all
$x$ of degree 1, and
\item[(b)]
$g$ and $g^{-1}$ send a homogeneous element of degree $v$ to a linear 
combination of homogeneous elements of degrees in $v {\mathbb N}$.
\end{enumerate}
Then $g$ restricts to an automorphism $g'$ of $A^{(v)}$. Further
$g$ is the identity if and only of $g'$ is.
\end{lemma}

\begin{proof} The first assertion is easy to show. Now we assume that
$g'$ is the identity. Then $g'(x^v)=x^v$ for all $x\in A$ of degree 1. 
This implies that
$$v\deg g(x)=\deg g(x)^v=\deg g(x^v)=\deg g'(x^v)=\deg x^v=v.$$
Hence $\deg g(x)=1$ and $g(x)=x$ by hypothesis (a).
\end{proof}

Now we are ready to prove the first Tits alternative theorem.

\begin{theorem}
\label{xxthm8.3} Suppose that $n$ is odd. 
\begin{enumerate}
\item[(1)]
If $\gcd(m,v)>1$, then $\Aut(\Bbbk_{q}[\bfx]^{(v)})$ is virtually abelian.
\item[(2)]
If $\gcd(m,v)=1$, then $\Aut(\Bbbk_{q}[\bfx]^{(v)})$ contains a free subgroup
of rank 2.
\end{enumerate}
\end{theorem}

\begin{proof} (1) This follows from Theorems \ref{xxthm0.1} and 
\ref{xxthm0.2}.

(2) Let $g_1$ and $g_3$ be the automorphisms of $\Bbbk_{q}[\bfx]$ given in 
Lemma \ref{xxlem8.1}. By Lemma \ref{xxlem8.1}(2),
the elements $g_1$ and $g_3$ generate a free subgroup of rank $2$,
by Lemma \ref{xxlem8.1}(3), they restrict to automorphisms $g'_1$ and $g'_3$ 
of $\Bbbk_{q}[\bfx]^{(v)}$.
We claim that $g'_1$ and $g'_3$ generates a free subgroup of
$\Aut(\Bbbk_{q}[\bfx]^{(v)})$. Let $\mathrm{Id}\neq g\in \langle g_1,g_3\rangle
\subseteq \Aut(\Bbbk_{q}[\bfx])$ and let $g'$ be the 
corresponding element in $\langle g'_1,g'_3\rangle 
\subseteq \Aut(\Bbbk_{q}[\bfx]^{(v)})$. By Lemma \ref{xxlem8.2}, the element
$g'$ is not the identity. Therefore $\langle g'_1,g'_3\rangle$ 
is a free group of rank 2.
\end{proof}

Secondly we consider the case when $n$ is even and write $n=2s$.
As before let $m$ be the order of $q$. We consider the case 
where $m$ divides $v$ and write $v=m\gamma$. 

\begin{lemma}
\label{xxlem8.4} 
Let $n=2$ and $v=m\gamma$ and $A=\Bbbk_{q}[\bfx]^{(v)}$.
Then $\Aut(A)$ contains a free subgroup of rank 2.
\end{lemma}

\begin{proof} By direct computation or
equations similar to \eqref{E3.0.4}-\eqref{E3.0.5}, 
if $n=2$, $A$ is isomorphic to the commutative ring 
$\Bbbk[x_1,x_2]^{(v)}$. So we identify
these two algebras. 

Consider two derivations
$$\partial_1: x_1\to x_2^{v+1}, \quad x_2\to 0$$
and
$$\partial_2: x_1\to 0, \quad x_2\to x_1^{v+1}.$$
Let $g_1=\exp(\partial_1)$ and $g_2=\exp(\partial_2)$.
Then, by \cite[Proposition 2.5]{CPWZ3},
$g_1$ and $g_2$ generate a free subgroup of rank $2$.
Since the degree of $\partial_i$ is $v$, we see that $g_1$ and
$g_2$ restrict to automorphisms $g'_1$ and $g'_2$ 
of $\Bbbk[x_1,x_2]^{(v)}$. By Lemma \ref{xxlem8.2},
the subgroup of $\Aut(A)$ generated $g'_1$ and $g'_2$
is free of rank $2$.
\end{proof}

\begin{proof}[Proof of Theorem \ref{xxthm0.7}]
When $n$ is odd, this follows from Theorem \ref{xxthm8.3}.
When $n=2$, this follows from Theorem \ref{xxthm0.2} and 
Lemma \ref{xxlem8.4}.
\end{proof}

\providecommand{\bysame}{\leavevmode\hbox to3em{\hrulefill}\thinspace}
\providecommand{\MR}{\relax\ifhmode\unskip\space\fi MR }
\providecommand{\MRhref}[2]{%

\href{http://www.ams.org/mathscinet-getitem?mr=#1}{#2} }
\providecommand{\href}[2]{#2}

\end{document}